\documentclass[leqno,12pt]{amsart}
\usepackage{amsfonts, amssymb,amsmath,amssymb}

\newtheorem{theorem}{Theorem}[section]

\newtheorem{proposition}[theorem]{Proposition}
\newtheorem{lemma}[theorem]{Lemma}

\theoremstyle{definition}
\newtheorem{definition}{Definition}[section]

\theoremstyle{remark}
\newtheorem{remark}{Remark}[section]

\def\supp{\mathop {\rm Supp}}
\def\Z{\ensuremath{\mathbb Z}}

\newcommand{\rankonecase}[1]{\smallbreak$\mathbf{#1}$}

\newcounter{primitivenumber}
\newcommand{\primitivetype}[1]{\setcounter{primitivenumber}{\theenumi}\end{enumerate}\subsubsection*{#1}\begin{enumerate}\setcounter{enumi}{\theprimitivenumber}}

\newsavebox{\aprime}
\newsavebox{\GreyCircle}
\newsavebox{\segm}
\newsavebox{\susp}
\newsavebox{\shortsusp}
\newsavebox{\bifurc}
\newsavebox{\longbifurc}
\newsavebox{\dthree}
\newsavebox{\dm}
\newsavebox{\shortdm}
\newsavebox{\atwo}
\newsavebox{\longam}
\newsavebox{\mediumam}
\newsavebox{\shortam}
\newsavebox{\plusaoneaone}
\newsavebox{\plusdm}
\newsavebox{\vsegm}

\newsavebox{\rightbisegm}
\newsavebox{\bsecondtwo}
\newsavebox{\shortbprimem}
\newsavebox{\btwo}
\newsavebox{\shortbm}
\newsavebox{\leftbisegm}
\newsavebox{\shortcm}
\newsavebox{\cprimetwo}
\newsavebox{\shortbsecondm}
\newsavebox{\shortcsecondm}
\newsavebox{\bthirdthree}
\newsavebox{\ffour}
\newsavebox{\lefttrisegm}
\newsavebox{\gtwo}
\newsavebox{\gprimetwo}
\newsavebox{\gsecondtwo}

\newsavebox{\GreyCircleTwo}
\newsavebox{\plusbm}
\newsavebox{\plusbprimem}
\newsavebox{\pluscsecondm}

\savebox{\aprime}{\begin{picture}(600,900)\put(300,900){\circle*{150}}\put(300,300){\circle{600}}\end{picture}}
\savebox{\GreyCircle}(600,600){\begin{picture}(600,600)(300,300)\put(600,600){\circle{600}}\put(30,25){\multiput(500,350)(150,0){2}{\circle*{70}}\multiput(425,425)(150,0){3}{\circle*{70}}\multiput(350,500)(150,0){4}{\circle*{70}}\multiput(425,575)(150,0){3}{\circle*{70}}\multiput(350,650)(150,0){4}{\circle*{70}}\multiput(425,725)(150,0){3}{\circle*{70}}\multiput(500,800)(150,0){2}{\circle*{70}}}\end{picture}}
\savebox{\segm}{\begin{picture}(1800,0)\multiput(0,0)(1800,0){2}{\circle*{150}}\thicklines\put(0,0){\line(1,0){1800}}\end{picture}}
\savebox{\susp}{\begin{picture}(3600,0)\multiput(0,0)(3600,0){2}{\circle*{150}}\thicklines\multiput(0,0)(2500,0){2}{\line(1,0){1100}}\multiput(1300,0)(400,0){3}{\line(1,0){200}}\end{picture}}
\savebox{\shortsusp}{\begin{picture}(1800,0)\multiput(0,0)(1800,0){2}{\circle*{150}}\thicklines\multiput(0,0)(400,0){5}{\line(1,0){200}}\end{picture}}
\savebox{\bifurc}{\begin{picture}(1200,2400)\multiput(1200,0)(0,2400){2}{\circle*{150}}\thicklines\put(0,1200){\line(1,1){1200}}\put(0,1200){\line(1,-1){1200}}\end{picture}}
\savebox{\longbifurc}{\begin{picture}(1800,3600)\multiput(1800,0)(0,3600){2}{\circle*{150}}\thicklines\put(0,1800){\line(1,1){1800}}\put(0,1800){\line(1,-1){1800}}\end{picture}}
\savebox{\dthree}{\begin{picture}(3600,600)\put(1500,0){\usebox{\GreyCircle}}\multiput(0,300)(1800,0){2}{\usebox{\segm}}\end{picture}}
\savebox{\dm}{\begin{picture}(8700,2400)\put(0,900){\usebox{\GreyCircle}}\multiput(300,1200)(5400,0){2}{\usebox{\segm}}\put(2100,1200){\usebox{\susp}}\put(7500,0){\usebox{\bifurc}}\end{picture}}
\savebox{\shortdm}{\begin{picture}(6900,2400)\put(0,900){\usebox{\GreyCircle}}\multiput(300,1200)(3600,0){2}{\usebox{\segm}}\put(2100,1200){\usebox{\shortsusp}}\put(5700,0){\usebox{\bifurc}}\end{picture}}
\savebox{\atwo}{\begin{picture}(2400,600)\put(300,300){\usebox{\segm}}\multiput(300,300)(1800,0){2}{\circle{600}}\multiput(300,300)(25,25){13}{\circle*{70}}\multiput(600,600)(300,0){4}{\multiput(0,0)(25,-25){7}{\circle*{70}}}\multiput(750,450)(300,0){4}{\multiput(0,0)(25,25){7}{\circle*{70}}}\multiput(2100,300)(-25,25){13}{\circle*{70}}\end{picture}}
\savebox{\longam}{\begin{picture}(7800,600)\put(300,300){\circle*{150}}\put(2100,300){\circle*{150}}\put(5700,300){\circle*{150}}\put(7500,300){\circle*{150}}\put(300,300){\circle{600}}\put(7500,300){\circle{600}}\multiput(300,300)(25,25){13}{\circle*{70}}\multiput(600,600)(3900,0){2}{\multiput(0,0)(300,0){9}{\multiput(0,0)(25,-25){7}{\circle*{70}}}\multiput(150,-150)(300,0){9}{\multiput(0,0)(25,25){7}{\circle*{70}}}}\multiput(7500,300)(-25,25){13}{\circle*{70}}\thicklines\put(300,300){\line( 1, 0){2250}}\put(7500,300){\line(-1, 0){2250}}\multiput(2850,300)(600,0){4}{\line( 1, 0){300}}\end{picture}}
\savebox{\mediumam}{\begin{picture}(6000,600)\multiput(300,300)(3600,0){2}{\usebox{\segm}}\put(2100,300){\usebox{\shortsusp}}\multiput(300,300)(1800,0){4}{\circle*{150}}\multiput(300,300)(5400,0){2}{\circle{600}}\multiput(300,300)(25,25){13}{\circle*{70}}\multiput(600,600)(3000,0){2}{\multiput(0,0)(300,0){6}{\multiput(0,0)(25,-25){7}{\circle*{70}}}\multiput(150,-150)(300,0){6}{\multiput(0,0)(25,25){7}{\circle*{70}}}}\multiput(5700,300)(-25,25){13}{\circle*{70}}\multiput(2700,600)(300,0){3}{\circle*{70}}\end{picture}}
\savebox{\shortam}{\begin{picture}(4200,600)\put(300,300){\usebox{\susp}}\multiput(300,300)(3600,0){2}{\circle{600}}\multiput(300,300)(25,25){13}{\circle*{70}}\multiput(600,600)(2100,0){2}{\multiput(0,0)(300,0){3}{\multiput(0,0)(25,-25){7}{\circle*{70}}}\multiput(150,-150)(300,0){3}{\multiput(0,0)(25,25){7}{\circle*{70}}}}\multiput(3900,300)(-25,25){13}{\circle*{70}}\multiput(1800,600)(300,0){3}{\circle*{70}}\end{picture}}
\savebox{\plusaoneaone}{\begin{picture}(1950,3000)\put(0,300){\usebox{\bifurc}}\multiput(1200,300)(0,2400){2}{\circle{600}}\multiput(1500,300)(0,2400){2}{\line(1,0){450}}\put(1950,300){\line(0,1){2400}}\end{picture}}
\savebox{\plusdm}{\begin{picture}(6600,2400)\put(0,1200){\usebox{\segm}}\put(1800,1200){\usebox{\susp}}\put(5400,0){\usebox{\bifurc}}\put(1500,900){\usebox{\GreyCircle}}\end{picture}}
\savebox{\vsegm}{\begin{picture}(0,1800)\multiput(0,0)(0,1800){2}{\circle*{150}}\thicklines\put(0,0){\line(0,1){1800}}\end{picture}}

\savebox{\GreyCircleTwo}{\begin{picture}(600,1200)(300,300)\put(600,600){\circle{600}}\put(30,25){\multiput(500,350)(150,0){2}{\circle*{70}}\multiput(425,425)(150,0){3}{\circle*{70}}\multiput(350,500)(150,0){4}{\circle*{70}}\multiput(425,575)(150,0){3}{\circle*{70}}\multiput(350,650)(150,0){4}{\circle*{70}}\multiput(425,725)(150,0){3}{\circle*{70}}\multiput(500,800)(150,0){2}{\circle*{70}}}\put(420,1020){\tiny2}\end{picture}}

\savebox{\rightbisegm}{\begin{picture}(1800,0)\multiput(0,0)(1800,0){2}{\circle*{150}}\thicklines\multiput(0,-60)(0,150){2}{\line(1,0){1800}}\multiput(1050,0)(-25,25){10}{\circle*{50}}\multiput(1050,0)(-25,-25){10}{\circle*{50}}\end{picture}}

\savebox{\bsecondtwo}{\begin{picture}(2400,600)\put(300,300){\usebox{\rightbisegm}}\put(0,0){\usebox{\GreyCircle}}\put(2100,300){\circle{600}}\end{picture}}

\savebox{\shortbprimem}{\begin{picture}(7500,1200)\multiput(300,300)(3600,0){2}{\usebox{\segm}}\put(2100,300){\usebox{\shortsusp}}\put(5700,300){\usebox{\rightbisegm}}\put(0,0){\usebox{\GreyCircleTwo}}\end{picture}}

\savebox{\btwo}{\begin{picture}(2100,600)\put(300,300){\usebox{\rightbisegm}}\put(0,0){\usebox{\GreyCircle}}\end{picture}}

\savebox{\shortbm}{\begin{picture}(7500,600)\multiput(300,300)(3600,0){2}{\usebox{\segm}}\put(2100,300){\usebox{\shortsusp}}\put(5700,300){\usebox{\rightbisegm}}\put(0,0){\usebox{\GreyCircle}}\end{picture}}

\savebox{\leftbisegm}{\begin{picture}(1800,0)\multiput(0,0)(1800,0){2}{\circle*{150}}\thicklines\multiput(0,-60)(0,150){2}{\line(1,0){1800}}\multiput(750,0)(25,25){10}{\circle*{50}}\multiput(750,0)(25,-25){10}{\circle*{50}}\end{picture}}

\savebox{\shortcm}{\begin{picture}(9000,600)\multiput(0,300)(1800,0){2}{\usebox{\segm}}\put(3600,300){\usebox{\shortsusp}}\put(5400,300){\usebox{\segm}}\put(7200,300){\usebox{\leftbisegm}}\put(1500,0){\usebox{\GreyCircle}}\end{picture}}

\savebox{\cprimetwo}{\begin{picture}(2100,1200)\put(0,300){\usebox{\leftbisegm}}\put(1500,0){\usebox{\GreyCircleTwo}}\end{picture}}

\savebox{\shortbsecondm}{\begin{picture}(7800,600)\multiput(300,300)(3600,0){2}{\usebox{\segm}}\put(2100,300){\usebox{\shortsusp}}\put(5700,300){\usebox{\rightbisegm}}\put(0,0){\usebox{\GreyCircle}}\put(7500,300){\circle{600}}\end{picture}}

\savebox{\shortcsecondm}{\begin{picture}(9300,600)\put(300,300){\circle{600}}\put(300,0){\usebox{\shortcm}}\end{picture}}

\savebox{\bthirdthree}{\begin{picture}(3900,600)\put(0,300){\usebox{\segm}}\put(1800,300){\usebox{\rightbisegm}}\put(3300,0){\usebox{\GreyCircle}}\end{picture}}

\savebox{\ffour}{\begin{picture}(5700,600)\multiput(0,300)(3600,0){2}{\usebox{\segm}}\put(1800,300){\usebox{\rightbisegm}}\put(5100,0){\usebox{\GreyCircle}}\end{picture}}

\savebox{\lefttrisegm}{\begin{picture}(1800,0)\multiput(0,0)(1800,0){2}{\circle*{150}}\thicklines\multiput(0,-120)(0,135){3}{\line(1,0){1800}}\multiput(750,0)(25,25){10}{\circle*{50}}\multiput(750,0)(25,-25){10}{\circle*{50}}\end{picture}}

\savebox{\gtwo}{\begin{picture}(2100,600)\put(300,300){\usebox{\lefttrisegm}}\put(0,0){\usebox{\GreyCircle}}\end{picture}}

\savebox{\gprimetwo}{\begin{picture}(2100,1200)\put(300,300){\usebox{\lefttrisegm}}\put(0,0){\usebox{\GreyCircleTwo}}\end{picture}}

\savebox{\gsecondtwo}{\begin{picture}(2400,600)\put(300,300){\usebox{\lefttrisegm}}\put(1800,0){\usebox{\GreyCircle}}\put(300,300){\circle{600}}\end{picture}}

\savebox{\plusbm}{\begin{picture}(9000,600)\put(0,300){\usebox{\segm}}\put(1500,0){\usebox{\shortbm}}\end{picture}}

\savebox{\plusbprimem}{\begin{picture}(9000,1200)\put(0,300){\usebox{\segm}}\put(1500,0){\usebox{\shortbprimem}}\end{picture}}

\savebox{\pluscsecondm}{\begin{picture}(9000,600)\put(0,0){\usebox{\shortcm}}\end{picture}}

\newcommand{\diagramBthree}{\begin{picture}(6000,600)\multiput(300,300)(1800,0){2}{\usebox{\segm}}\put(3900,300){\usebox{\rightbisegm}}\multiput(0,0)(5400,0){2}{\usebox{\GreyCircle}}\end{picture}}

\newcommand{\diagramBsix}{\begin{picture}(14700,600)\put(0,0){\usebox{\mediumam}}\put(5700,300){\usebox{\segm}}\put(7200,0){\usebox{\shortbm}}\end{picture}}

\newcommand{\diagramBsixbis}{\begin{picture}(14700,1200)\put(0,0){\usebox{\mediumam}}\put(5700,300){\usebox{\segm}}\put(7200,0){\usebox{\shortbprimem}}\end{picture}}

\newcommand{\diagramCone}{\begin{picture}(12000,1050)\put(300,750){\circle*{150}}\put(300,750){\circle{600}}\multiput(300,0)(2700,0){2}{\line(0,1){450}}\put(300,0){\line(1,0){2700}}\put(2700,450){\usebox{\shortcsecondm}}\end{picture}}

\newcommand{\diagramCfourprime}{\begin{picture}(9600,1200)\put(300,750){\usebox{\segm}}\put(1800,450){\usebox{\mediumam}}\multiput(300,750)(9000,0){2}{\circle{600}}\multiput(300,450)(9000,0){2}{\line(0,-1){450}}\put(300,0){\line(1,0){9000}}\put(7500,750){\usebox{\leftbisegm}}\end{picture}}

\newcommand{\diagramCfive}{\begin{picture}(9300,1200)\put(0,0){\usebox{\aprime}}\put(300,600){\usebox{\shortcm}}\end{picture}}

\newcommand{\diagramDfour}{\begin{picture}(14100,2400)\put(0,900){\usebox{\mediumam}}\put(5700,1200){\usebox{\segm}}\put(7200,0){\usebox{\shortdm}}\end{picture}}

\newcommand{\diagramDseven}{\begin{picture}(7200,3000)\multiput(300,1500)(3600,0){2}{\usebox{\segm}}\put(2100,1500){\usebox{\shortsusp}}\put(5700,300){\usebox{\bifurc}}\put(300,1500){\circle{600}}\multiput(6900,300)(0,2400){2}{\circle{600}}\put(0,1200){\multiput(300,300)(25,25){13}{\circle*{70}}\put(600,600){\multiput(0,0)(300,0){6}{\multiput(0,0)(25,-25){7}{\circle*{70}}}\multiput(150,-150)(300,0){6}{\multiput(0,0)(25,25){7}{\circle*{70}}}}\put(3600,600){\multiput(0,0)(300,0){7}{\multiput(0,0)(25,-25){7}{\circle*{70}}}\multiput(150,-150)(300,0){6}{\multiput(0,0)(25,25){7}{\circle*{70}}}}\multiput(2700,600)(300,0){3}{\circle*{70}}\multiput(300,300)(25,-25){13}{\circle*{70}}\put(600,0){\multiput(0,0)(300,0){6}{\multiput(0,0)(25,25){7}{\circle*{70}}}\multiput(150,150)(300,0){6}{\multiput(0,0)(25,-25){7}{\circle*{70}}}}\put(3600,0){\multiput(0,0)(300,0){7}{\multiput(0,0)(25,25){7}{\circle*{70}}}\multiput(150,150)(300,0){6}{\multiput(0,0)(25,-25){7}{\circle*{70}}}}\multiput(2700,0)(300,0){3}{\circle*{70}}}\put(3600,0){\thicklines\put(3300,2700){\line(-1,0){400}}\multiput(2900,2700)(-200,-200){5}{\line(0,-1){200}}\multiput(2900,2500)(-200,-200){4}{\line(-1,0){200}}\multiput(2100,1700)(-30,-10){5}{\line(-1,0){30}}\put(3300,300){\line(-1,0){400}}\multiput(2900,300)(-200,200){5}{\line(0,1){200}}\multiput(2900,500)(-200,200){4}{\line(-1,0){200}}\multiput(2100,1300)(-30,10){5}{\line(-1,0){30}}}\end{picture}}

\newcommand{\diagramEone}{\begin{picture}(7800,2100)\multiput(300,1800)(1800,0){4}{\usebox{\segm}}\put(3900,0){\usebox{\vsegm}}\multiput(0,1500)(7200,0){2}{\usebox{\GreyCircle}}\end{picture}}

\newcommand{\diagramEtwo}{\begin{picture}(7800,2400)\multiput(300,2100)(1800,0){4}{\usebox{\segm}}\put(3900,300){\usebox{\vsegm}}\multiput(300,2100)(7200,0){2}{\circle{600}}\put(3600,0){\usebox{\GreyCircle}}\multiput(300,2100)(25,25){13}{\circle*{70}}\put(600,2400){\multiput(0,0)(300,0){22}{\multiput(0,0)(25,-25){7}{\circle*{70}}}\multiput(150,-150)(300,0){22}{\multiput(0,0)(25,25){7}{\circle*{70}}}}\multiput(7500,2100)(-25,25){13}{\circle*{70}}\end{picture}}

\newcommand{\diagramFone}{\begin{picture}(6000,1050)\multiput(300,750)(3600,0){2}{\usebox{\segm}}\multiput(300,750)(5400,0){2}{\circle{600}}\put(1800,450){\usebox{\bsecondtwo}}\multiput(300,0)(5400,0){2}{\line(0,1){450}}\put(300,0){\line(1,0){5400}}\end{picture}}

\newcommand{\diagramFthree}{\begin{picture}(6000,600)\multiput(300,300)(3600,0){2}{\usebox{\segm}}\put(2100,300){\usebox{\rightbisegm}}\multiput(0,0)(3600,0){2}{\usebox{\GreyCircle}}\put(5700,300){\circle{600}}\end{picture}}

\newcommand{\diagramGone}{\begin{picture}(2400,1050)\put(300,900){\usebox{\lefttrisegm}}\multiput(0,0)(1800,0){2}{\usebox{\aprime}}\end{picture}}

\newcommand{\diagramGtwo}{\begin{picture}(6900,1500)\multiput(0,0)(4500,0){2}{\put(300,750){\usebox{\lefttrisegm}}\multiput(300,750)(1800,0){2}{\circle{600}}\put(300,0){\line(0,1){450}}\put(2100,1050){\line(0,1){450}}}\multiput(300,0)(1800,1500){2}{\line(1,0){4500}}\end{picture}}

\newcommand{\diagramaapp}{\begin{picture}(7800,3300)\multiput(0,0)(0,2700){2}{\multiput(300,300)(5400,0){2}{\usebox{\segm}}\put(2100,300){\usebox{\susp}}\multiput(300,300)(1800,0){2}{\circle{600}}\multiput(5700,300)(1800,0){2}{\circle{600}}}\multiput(300,600)(1800,0){2}{\line(0,1){2100}}\multiput(5700,600)(1800,0){2}{\line(0,1){2100}}\end{picture}}

\newcommand{\diagramacn}{\begin{picture}(10800,600)\put(0,0){\usebox{\dthree}}\put(3600,300){\usebox{\susp}}\put(7200,0){\usebox{\dthree}}\end{picture}}

\newcommand{\diagramaapplusqplusp}{\begin{picture}(16800,1500)\multiput(300,1200)(12600,0){2}{\usebox{\susp}}\multiput(3900,1200)(7200,0){2}{\usebox{\segm}}\put(5400,900){\usebox{\mediumam}}\multiput(300,1200)(16200,0){2}{\circle{600}}\multiput(3900,1200)(9000,0){2}{\circle{600}}\multiput(3900,900)(9000,0){2}{\line(0,-1){450}}\put(3900,450){\line(1,0){9000}}\multiput(300,900)(16200,0){2}{\line(0,-1){900}}\put(300,0){\line(1,0){16200}}\end{picture}}

\newcommand{\diagramaaastpplusoneplusp}{\begin{picture}(11400,1800)\multiput(300,1500)(7200,0){2}{\usebox{\susp}}\multiput(3900,1500)(1800,0){2}{\usebox{\segm}}\put(5400,600){\usebox{\aprime}}\multiput(300,1500)(10800,0){2}{\circle{600}}\multiput(3900,1500)(3600,0){2}{\circle{600}}\multiput(3900,1200)(3600,0){2}{\line(0,-1){900}}\put(3900,300){\line(1,0){3600}}\multiput(300,1200)(10800,0){2}{\line(0,-1){1200}}\put(300,0){\line(1,0){10800}}\end{picture}}

\newcommand{\diagramaon}{\begin{picture}(7800,900)\multiput(300,900)(5400,0){2}{\usebox{\segm}}\put(2100,900){\usebox{\susp}}\multiput(0,0)(5400,0){2}{\multiput(0,0)(1800,0){2}{\usebox{\aprime}}}\end{picture}}

\newcommand{\diagramacastn}{\begin{picture}(11400,600)\multiput(0,0)(7200,0){2}{\multiput(0,0)(1800,0){2}{\usebox{\atwo}}}\put(3900,300){\usebox{\susp}}\multiput(3900,300)(25,25){13}{\circle*{70}}\multiput(4200,600)(2400,0){2}{\multiput(0,0)(300,0){2}{\multiput(0,0)(25,-25){7}{\circle*{70}}}\multiput(150,-150)(300,0){2}{\multiput(0,0)(25,25){7}{\circle*{70}}}}\multiput(7500,300)(-25,25){13}{\circle*{70}}\end{picture}}

\newcommand{\diagrambbpp}{\begin{picture}(9600,3300)\multiput(0,0)(0,2700){2}{\multiput(300,300)(5400,0){2}{\usebox{\segm}}\put(7500,300){\usebox{\rightbisegm}}\put(2100,300){\usebox{\susp}}\multiput(300,300)(1800,0){2}{\circle{600}}\multiput(5700,300)(1800,0){3}{\circle{600}}}\multiput(300,600)(1800,0){2}{\line(0,1){2100}}\multiput(5700,600)(1800,0){3}{\line(0,1){2100}}\end{picture}}

\newcommand{\diagrambcastn}{\begin{picture}(11400,600)\multiput(0,0)(1800,0){2}{\usebox{\atwo}}\put(7200,0){\usebox{\atwo}}\put(9000,0){\usebox{\bsecondtwo}}\put(3900,300){\usebox{\susp}}\multiput(3900,300)(25,25){13}{\circle*{70}}\multiput(4200,600)(2400,0){2}{\multiput(0,0)(300,0){2}{\multiput(0,0)(25,-25){7}{\circle*{70}}}\multiput(150,-150)(300,0){2}{\multiput(0,0)(25,25){7}{\circle*{70}}}}\multiput(7500,300)(-25,25){13}{\circle*{70}}\end{picture}}

\newcommand{\diagrambopplusq}{\begin{picture}(14700,1800)\multiput(300,900)(5400,0){2}{\usebox{\segm}}\put(2100,900){\usebox{\susp}}\multiput(0,0)(1800,0){2}{\usebox{\aprime}}\put(5400,0){\usebox{\aprime}}\put(7200,600){\usebox{\shortbprimem}}\end{picture}}

\newcommand{\diagramacastpplusbprimeq}{\begin{picture}(18300,1200)\multiput(0,0)(1800,0){2}{\usebox{\atwo}}\put(7200,0){\usebox{\atwo}}\put(9300,300){\usebox{\segm}}\put(3900,300){\usebox{\susp}}\multiput(3900,300)(25,25){13}{\circle*{70}}\multiput(4200,600)(2400,0){2}{\multiput(0,0)(300,0){2}{\multiput(0,0)(25,-25){7}{\circle*{70}}}\multiput(150,-150)(300,0){2}{\multiput(0,0)(25,25){7}{\circle*{70}}}}\multiput(7500,300)(-25,25){13}{\circle*{70}}\put(10800,0){\usebox{\shortbprimem}}\end{picture}}

\newcommand{\diagrambcprimen}{\begin{picture}(11400,1200)\put(0,600){\multiput(0,0)(1800,0){2}{\usebox{\atwo}}\put(7200,0){\usebox{\atwo}}\put(3900,300){\usebox{\susp}}\multiput(3900,300)(25,25){13}{\circle*{70}}\multiput(4200,600)(2400,0){2}{\multiput(0,0)(300,0){2}{\multiput(0,0)(25,-25){7}{\circle*{70}}}\multiput(150,-150)(300,0){2}{\multiput(0,0)(25,25){7}{\circle*{70}}}}\multiput(7500,300)(-25,25){13}{\circle*{70}}}\put(9000,600){\usebox{\btwo}}\put(10800,0){\usebox{\aprime}}\end{picture}}

\newcommand{\diagramacastpplusbq}{\begin{picture}(18300,600)\multiput(0,0)(1800,0){2}{\usebox{\atwo}}\put(7200,0){\usebox{\atwo}}\put(9300,300){\usebox{\segm}}\put(3900,300){\usebox{\susp}}\multiput(3900,300)(25,25){13}{\circle*{70}}\multiput(4200,600)(2400,0){2}{\multiput(0,0)(300,0){2}{\multiput(0,0)(25,-25){7}{\circle*{70}}}\multiput(150,-150)(300,0){2}{\multiput(0,0)(25,25){7}{\circle*{70}}}}\multiput(7500,300)(-25,25){13}{\circle*{70}}\put(10800,0){\usebox{\shortbm}}\end{picture}}

\newcommand{\diagramccpplusq}{\begin{picture}(19800,600)\put(0,0){\usebox{\dthree}}\put(3600,300){\usebox{\susp}}\put(7200,0){\usebox{\dthree}}\put(10800,0){\usebox{\shortcm}}\end{picture}}

\newcommand{\diagramccprimepplustwo}{\begin{picture}(12900,1200)\put(0,0){\usebox{\dthree}}\put(3600,300){\usebox{\susp}}\put(7200,0){\usebox{\dthree}}\put(10800,0){\usebox{\cprimetwo}}\end{picture}}

\newcommand{\diagramccpp}{\begin{picture}(9600,3300)\multiput(0,0)(0,2700){2}{\multiput(300,300)(5400,0){2}{\usebox{\segm}}\put(7500,300){\usebox{\leftbisegm}}\put(2100,300){\usebox{\susp}}\multiput(300,300)(1800,0){2}{\circle{600}}\multiput(5700,300)(1800,0){3}{\circle{600}}}\multiput(300,600)(1800,0){2}{\line(0,1){2100}}\multiput(5700,600)(1800,0){3}{\line(0,1){2100}}\end{picture}}

\newcommand{\diagramcon}{\begin{picture}(9600,1050)\multiput(300,900)(5400,0){2}{\usebox{\segm}}\put(2100,900){\usebox{\susp}}\multiput(0,0)(1800,0){2}{\usebox{\aprime}}\multiput(5400,0)(1800,0){3}{\usebox{\aprime}}\put(7500,900){\usebox{\leftbisegm}}\end{picture}}

\newcommand{\diagramaaonepluspplusonepluscsecondq}{\begin{picture}(18300,1200)\multiput(300,750)(7200,0){2}{\usebox{\segm}}\put(1800,450){\usebox{\mediumam}}\multiput(300,750)(9000,0){2}{\circle{600}}\multiput(300,450)(9000,0){2}{\line(0,-1){450}}\put(300,0){\line(1,0){9000}}\put(9300,450){\usebox{\shortcm}}\end{picture}}

\newcommand{\diagramaaoneonepluscsecondnonepluscsecondntwo}{\begin{picture}(9300,3300)\multiput(0,0)(0,2700){2}{\put(300,300){\circle{600}}\put(300,0){\usebox{\shortcm}}}\put(300,600){\line(0,1){2100}}\end{picture}}

\newcommand{\diagramacastppluscsecondq}{\begin{picture}(18300,600)\multiput(0,0)(1800,0){2}{\usebox{\atwo}}\put(7200,0){\usebox{\atwo}}
\put(3900,300){\usebox{\susp}}\multiput(3900,300)(25,25){13}{\circle*{70}}\multiput(4200,600)(2400,0){2}{\multiput(0,0)(300,0){2}{\multiput(0,0)(25,-25){7}{\circle*{70}}}\multiput(150,-150)(300,0){2}{\multiput(0,0)(25,25){7}{\circle*{70}}}}\multiput(7500,300)(-25,25){13}{\circle*{70}}
\put(9300,0){\usebox{\shortcm}}\end{picture}}

\newcommand{\diagramdcprimen}{\begin{picture}(14100,3300)\multiput(0,1800)(7200,0){2}{\usebox{\dthree}}\put(3600,2100){\usebox{\susp}}\put(10800,2100){\usebox{\segm}}\put(12600,900){\usebox{\bifurc}}\put(12300,1800){\usebox{\GreyCircle}}\put(13500,0){\usebox{\aprime}}\end{picture}}

\newcommand{\diagramdopplusq}{\begin{picture}(14100,2400)\multiput(0,300)(5400,0){2}{\put(0,0){\usebox{\aprime}}\put(300,900){\usebox{\segm}}}\put(1800,300){\usebox{\aprime}}\put(7200,0){\usebox{\shortdm}}\put(2100,1200){\usebox{\susp}}
\end{picture}}

\newcommand{\diagramdcnodd}{\begin{picture}(12300,3000)\put(0,1200){\usebox{\dthree}}\put(3600,1500){\usebox{\susp}}\put(7200,1200){\usebox{\dthree}}\put(10800,300){\usebox{\bifurc}}\multiput(12000,300)(0,2400){2}{\circle{600}}\thicklines\multiput(12000,300)(0,2000){2}{\line(0,1){400}}\multiput(12000,700)(-200,200){4}{\line(-1,0){200}}\multiput(11800,700)(-200,200){4}{\line(0,1){200}}\multiput(12000,2300)(-200,-200){4}{\line(-1,0){200}}\multiput(11800,2300)(-200,-200){4}{\line(0,-1){200}}\end{picture}}

\newcommand{\diagramddpp}{\begin{picture}(9900,5700)\put(0,1800){\multiput(300,1200)(900,-2700){2}{\multiput(0,1200)(5400,0){2}{\usebox{\segm}}\put(1800,1200){\usebox{\susp}}\put(7200,0){\usebox{\bifurc}}\multiput(0,1200)(5400,0){2}{\multiput(0,0)(1800,0){2}{\circle{600}}}\multiput(8400,0)(0,2400){2}{\circle{600}}}\put(400,900){\multiput(0,1200)(5400,0){2}{\multiput(0,0)(1800,0){2}{\line(1,-3){700}}}\multiput(8400,0)(0,2400){2}{\line(1,-3){700}}}}\end{picture}}

\newcommand{\diagramdon}{\begin{picture}(9000,3300)\multiput(0,1200)(5400,0){2}{\multiput(0,0)(1800,0){2}{\usebox{\aprime}}}\multiput(300,2100)(5400,0){2}{\usebox{\segm}}\put(2100,2100){\usebox{\susp}}\put(7500,900){\usebox{\bifurc}}\multiput(8400,0)(0,2400){2}{\usebox{\aprime}}\end{picture}}

\newcommand{\diagramdsastfour}{\begin{picture}(3600,3000)\put(300,1500){\usebox{\segm}}\put(2100,300){\usebox{\bifurc}}\put(300,1500){\circle{600}}\multiput(3300,300)(0,2400){2}{\circle{600}}\multiput(300,1500)(25,25){13}{\circle*{70}}\multiput(600,1800)(300,0){5}{\multiput(0,0)(25,-25){7}{\circle*{70}}}\multiput(750,1650)(300,0){4}{\multiput(0,0)(25,25){7}{\circle*{70}}}\multiput(300,1500)(25,-25){13}{\circle*{70}}\multiput(600,1200)(300,0){5}{\multiput(0,0)(25,25){7}{\circle*{70}}}\multiput(750,1350)(300,0){4}{\multiput(0,0)(25,-25){7}{\circle*{70}}}\thicklines\put(3300,2700){\line(-1,0){400}}\multiput(2900,2700)(-200,-200){5}{\line(0,-1){200}}\multiput(2900,2500)(-200,-200){4}{\line(-1,0){200}}\multiput(2100,1700)(-30,-10){5}{\line(-1,0){30}}\put(3300,300){\line(-1,0){400}}\multiput(2900,300)(-200,200){5}{\line(0,1){200}}\multiput(2900,500)(-200,200){4}{\line(-1,0){200}}\multiput(2100,1300)(-30,10){5}{\line(-1,0){30}}\multiput(3300,300)(0,2000){2}{\line(0,1){400}}\multiput(3300,700)(-200,200){4}{\line(-1,0){200}}\multiput(3100,700)(-200,200){4}{\line(0,1){200}}\multiput(3300,2300)(-200,-200){4}{\line(-1,0){200}}\multiput(3100,2300)(-200,-200){4}{\line(0,-1){200}}\end{picture}}

\newcommand{\diagramdcastn}{\begin{picture}(10800,3000)\multiput(0,1200)(1800,0){2}{\usebox{\atwo}}\put(7200,1200){\usebox{\atwo}}\put(3900,1500){\usebox{\susp}}\put(9300,300){\usebox{\bifurc}}\multiput(10500,300)(0,2400){2}{\circle{600}}\multiput(3900,1500)(25,25){13}{\circle*{70}}\multiput(4200,1800)(2400,0){2}{\multiput(0,0)(300,0){2}{\multiput(0,0)(25,-25){7}{\circle*{70}}}\multiput(150,-150)(300,0){2}{\multiput(0,0)(25,25){7}{\circle*{70}}}}\multiput(7500,1500)(-25,25){13}{\circle*{70}}\thicklines\put(9300,1500){\line(0,1){400}}\multiput(9300,1900)(200,200){4}{\line(1,0){200}}\multiput(9500,1900)(200,200){4}{\line(0,1){200}}\put(10500,2700){\line(-1,0){400}}\put(9300,1500){\line(1,0){400}}\multiput(9700,1500)(200,-200){4}{\line(0,-1){200}}\multiput(9700,1300)(200,-200){4}{\line(1,0){200}}\put(10500,300){\line(0,1){400}}\end{picture}}

\newcommand{\diagramacastpplusdq}{\begin{picture}(17700,2400)\put(0,900){\multiput(0,0)(7200,0){2}{\put(0,0){\usebox{\atwo}}}\put(1800,0){\usebox{\atwo}}\put(3900,300){\usebox{\susp}}\multiput(3900,300)(25,25){13}{\circle*{70}}\multiput(4200,600)(2400,0){2}{\multiput(0,0)(300,0){2}{\multiput(0,0)(25,-25){7}{\circle*{70}}}\multiput(150,-150)(300,0){2}{\multiput(0,0)(25,25){7}{\circle*{70}}}}\multiput(7500,300)(-25,25){13}{\circle*{70}}}\put(9300,1200){\usebox{\segm}}\put(10800,0){\usebox{\shortdm}}\end{picture}}

\newcommand{\diagramefprimen}{\begin{picture}(11400,2100)\multiput(300,1800)(1800,0){5}{\usebox{\segm}}\put(3900,0){\usebox{\vsegm}}\multiput(0,1500)(7200,0){2}{\usebox{\GreyCircle}}\multiput(9000,900)(1800,0){2}{\usebox{\aprime}}\put(9300,1800){\usebox{\shortsusp}}\end{picture}}

\newcommand{\diagramecseven}{\begin{picture}(9300,2100)\multiput(300,1800)(1800,0){5}{\usebox{\segm}}\put(3900,0){\usebox{\vsegm}}\multiput(3600,1500)(3600,0){2}{\usebox{\GreyCircle}}\multiput(0,900)(1800,0){2}{\usebox{\aprime}}\end{picture}}

\newcommand{\diagrameepp}{\begin{picture}(12450,5550)\put(1050,0){\multiput(0,0)(-1050,3150){2}{\multiput(300,2100)(1800,0){4}{\usebox{\segm}}\put(3900,300){\usebox{\vsegm}}\multiput(7500,2100)(1800,0){2}{\usebox{\shortsusp}}\multiput(300,2100)(1800,0){7}{\circle{600}}\put(3900,300){\circle{600}}}\put(-100,300){\multiput(300,2100)(1800,0){5}{\line(-1,3){850}}\put(3900,300){\line(-1,3){850}}\multiput(9300,2100)(1800,0){2}{\multiput(0,-50)(-50,150){18}{\line(0,1){100}}}}}\end{picture}}

\newcommand{\diagrameasix}{\begin{picture}(7800,3600)\multiput(300,2700)(1800,0){4}{\usebox{\segm}}\put(3900,900){\usebox{\vsegm}}\multiput(3600,0)(0,1800){2}{\usebox{\aprime}}\multiput(300,2700)(5400,0){2}{\multiput(0,0)(1800,0){2}{\circle{600}}}\multiput(300,3000)(7200,0){2}{\line(0,1){600}}\put(300,3600){\line(1,0){7200}}\multiput(2100,3000)(3600,0){2}{\line(0,1){300}}\put(2100,3300){\line(1,0){3600}}\end{picture}}

\newcommand{\diagrameon}{\begin{picture}(11400,3000)\multiput(300,2700)(1800,0){4}{\usebox{\segm}}\put(3900,900){\usebox{\vsegm}}\multiput(0,1800)(1800,0){7}{\usebox{\aprime}}\put(3600,0){\usebox{\aprime}}\multiput(7500,2700)(1800,0){2}{\usebox{\shortsusp}}\end{picture}}

\newcommand{\diagramecastn}{\begin{picture}(11400,2400)\multiput(300,2100)(1800,0){4}{\usebox{\segm}}\put(3900,300){\usebox{\vsegm}}\multiput(7500,2100)(1800,0){2}{\usebox{\shortsusp}}\multiput(0,1800)(1800,0){6}{\multiput(300,300)(1800,0){2}{\circle{600}}\multiput(300,300)(25,25){13}{\circle*{70}}\multiput(600,600)(300,0){4}{\multiput(0,0)(25,-25){7}{\circle*{70}}}\multiput(750,450)(300,0){4}{\multiput(0,0)(25,25){7}{\circle*{70}}}\multiput(2100,300)(-25,25){13}{\circle*{70}}}\put(3900,300){\circle{600}}\put(3600,0){\multiput(300,300)(25,25){13}{\circle*{70}}\multiput(600,600)(0,300){4}{\multiput(0,0)(-25,25){7}{\circle*{70}}}\multiput(450,750)(0,300){4}{\multiput(0,0)(25,25){7}{\circle*{70}}}\multiput(300,2100)(25,-25){13}{\circle*{70}}}\end{picture}}

\newcommand{\diagramefsixplusatwo}{\begin{picture}(11400,2100)\multiput(300,1800)(1800,0){4}{\usebox{\segm}}\put(3900,0){\usebox{\vsegm}}\multiput(0,1500)(7200,0){2}{\usebox{\GreyCircle}}\put(7500,1800){\usebox{\segm}}\put(9000,1500){\usebox{\atwo}}\end{picture}}

\newcommand{\diagramaatwotwoplusatwo}{\begin{picture}(7800,3000)\put(0,-600){\multiput(300,2700)(1800,0){4}{\usebox{\segm}}\put(3900,900){\usebox{\vsegm}}\multiput(300,2700)(5400,0){2}{\multiput(0,0)(1800,0){2}{\circle{600}}}\multiput(300,3000)(7200,0){2}{\line(0,1){600}}\put(300,3600){\line(1,0){7200}}\multiput(2100,3000)(3600,0){2}{\line(0,1){300}}\put(2100,3300){\line(1,0){3600}}}\put(3600,0){\multiput(300,300)(0,1800){2}{\circle{600}}\multiput(300,300)(25,25){13}{\circle*{70}}\multiput(600,600)(0,300){4}{\multiput(0,0)(-25,25){7}{\circle*{70}}}\multiput(450,750)(0,300){4}{\multiput(0,0)(25,25){7}{\circle*{70}}}\multiput(300,2100)(25,-25){13}{\circle*{70}}}\end{picture}}

\newcommand{\diagramacfiveplusatwo}{\begin{picture}(9300,2100)\put(0,1500){\usebox{\atwo}}\put(2100,0){\multiput(0,1800)(1800,0){4}{\usebox{\segm}}\put(1800,0){\usebox{\vsegm}}\multiput(1500,1500)(3600,0){2}{\usebox{\GreyCircle}}}\end{picture}}

\newcommand{\diagramfffourfour}{\begin{picture}(6000,3300)\multiput(0,0)(0,2700){2}{\multiput(300,300)(3600,0){2}{\usebox{\segm}}\put(2100,300){\usebox{\rightbisegm}}\multiput(300,300)(1800,0){4}{\circle{600}}}\multiput(300,600)(1800,0){4}{\line(0,1){2100}}\end{picture}}

\newcommand{\diagramfofour}{\begin{picture}(6000,900)\multiput(300,900)(3600,0){2}{\usebox{\segm}}\put(2100,900){\usebox{\rightbisegm}}\multiput(0,0)(1800,0){4}{\usebox{\aprime}}\end{picture}}

\newcommand{\diagramaotwoplusatwo}{\begin{picture}(6000,1200)\multiput(300,900)(3600,0){2}{\usebox{\segm}}\put(2100,900){\usebox{\rightbisegm}}\multiput(3600,0)(1800,0){2}{\usebox{\aprime}}\put(0,600){\usebox{\atwo}}\end{picture}}

\newcommand{\diagramfcastfour}{\begin{picture}(6000,600)\multiput(0,0)(3600,0){2}{\usebox{\atwo}}\put(1800,0){\usebox{\btwo}}\end{picture}}

\begin{document}

\title[Classification of strict wonderful varieties]
{Classification of strict wonderful varieties}
\author[P. Bravi and S. Cupit-Foutou ]{P. Bravi and S. Cupit-Foutou}

\address{Paolo Bravi, Dipartimento di Matematica, Universit\`a di Padova, Italy, 
\textit{Currently:} Dipartimento di Matematica, Universit\`a di Roma La Sapienza, P.le Aldo Moro 5, 00185 Roma, Italy}
\email{bravi@mat.uniroma1.it}

\address{St\'ephanie Cupit-Foutou, Mathematisches Institut, Universit\"at zu K\"oln, Weyertal Str.\ 86-90, 
50931 K\"oln, Germany}
\email{scupit@math.uni-koeln.de}

\begin{abstract}
In the setting of strict wonderful varieties we prove Luna's conjecture,
saying that wonderful varieties are classified by combinatorial objects, the so-called spherical systems.
In particular, we prove that primitive strict wonderful varieties are mostly obtained from symmetric spaces, spherical nilpotent
orbits and model spaces. To make the paper as self-contained as possible, we also gather some known results on these families and more generally on wonderful varieties.  
\end{abstract}

\maketitle

\tableofcontents

\section{Introduction}
First examples of wonderful varieties appeared in enumerative geometry
with the so-called variety of complete quadrics in the complex
projective space $\mathbb P^n$.
This variety is a compactification of the set of non-degenerate
quadrics as an algebraic subset of some given projective  space.
It has beautiful properties like being smooth, having finitely many orbits under the action
of the automorphism group $PSL_{n+1}$ of $\mathbb P^n$...
The set of non-degenerate quadrics is isomorphic to the symmetric space $PSL_{n+1}/PSO_{n+1}$. 
Generalisations of the variety of complete quadrics were constructed
and studied by De Concini and Procesi in \cite{DP83} as they considered more general symmetric spaces.

In \cite{LV}, Luna and Vust developed  a general theory on embeddings of homogeneous spaces for algebraic groups.
This theory is particularly well-developed in the context of spherical algebraic varieties (e.g.\ toric varieties, flag varieties, symmetric varieties...) and can be seen as a generalisation of the combinatorics of toric varieties.

Complete embeddings sharing the properties of De Concini-Procesi compactifications are called wonderful varieties.
Luna proved that wonderful varieties were spherical (\cite{L96}) and that they play a central role in the study of spherical varieties (\cite{Lu01}).
After rank 1 and 2 wonderful varieties were classified (by
Ahiezer \cite{Ah} and Wasserman \cite{W} respectively), 
Luna introduced in \cite{Lu01} combinatorial invariants, \textit{the spherical systems}, that he was able to attach to any wonderful variety.
Conversely, he conjectured in loc.\ cit.\ that wonderful varieties could be classified by these spherical systems. 
Only for particular acting groups, positive answers to this conjecture
have been obtained (\textit{see} loc.\ cit., \cite {BP,Bra}). For any
group, the fact
that non-isomorphic wonderful varieties have different spherical
systems has recently been proved by Losev (\cite{Lo}, \textit{see}
also Section~\ref{uniqueness}).

To any spherical homogeneous space, say $G/H$ with $G$ a reductive
algebraic group, one can naturally assign a wonderful variety:
take the wonderful embedding of  $G/N_G(H)$ where $N_G(H)$ is the
normaliser of $H$ in $G$. 
This wonderful embedding exists by a result of Knop (\cite{Kn96}) and
it is unique.

Here we are interested in wonderful varieties whose every point has
a selfnormalising stabiliser.
This property can be read off the associated spherical system.  
Spherical systems and wonderful varieties satisfying this property are
called \textit{strict}. 

In this paper, we positively answer to Luna's conjecture for a
slightly more general class of spherical systems: we prove that to
every such spherical system, there corresponds a (unique) wonderful
variety.

It is known that for many purposes the study of wonderful varieties
can be reduced to that of the so-called primitive wonderful varieties.

We show that most of the strict primitive wonderful varieties
come either from an affine spherical homogeneous space (and often in
particular from a symmetric variety), from a spherical nilpotent adjoint orbit or from a model homogeneous space.

This paper is an expansion of Section~6 of \cite{BCF06}. 

\smallbreak\noindent{\textbf{Organisation of the paper.}}
In the second section, we gather definitions and results related to wonderful and spherical varieties. We also introduce \textit{strict} wonderful varieties.
Some classes of examples of spherical varieties are given;
their properties and their classification are recalled with more details left to the corresponding appendices.
In the third section, following Luna, we introduce several combinatorial invariants: \textit{strict} spherical systems and their colours.
In the fourth section we state Luna's dictionary which relates some geometrical properties of wonderful varieties
to combinatorial properties of their spherical systems.
The fifth section is devoted to the main result of this paper; we give
in particular the list of all primitive spherical systems without
simple spherical roots 
with their corresponding wonderful subgroups.
The last section is dedicated to the proofs of our results.
 
\smallbreak\noindent{\textbf{Acknowledgments.}}
Both authors would like to thank P.~Littelmann and D.~Luna for helpful
discussions, suggestions and support. 
They are also grateful to the anonymous referee 
for his careful reading and valuable comments.

\smallbreak\noindent{\textbf{Main notation.}}
In the following, $G$ is a connected reductive complex algebraic group and $T$ is a maximal torus of $G$.
We fix a Borel subgroup $B$ containing $T$ and denote by $B^-$ its opposite, that is $B\cap B^-=T$. 
Let $\Phi$ be the root system of $G$ corresponding to $T$ and $S$ the set of simple roots associated to $B$ (numbered as in Bourbaki, \cite{Bo}). The coroot of the root $\alpha$ is denoted by $\alpha^\vee$.

\section{Wonderful varieties}

\subsection{Definitions}

In this section we freely recall from \cite{Kn91,Bri97,Lu01} notions and results on spherical varieties
with particular attention to wonderful varieties.

\begin{definition}
An (irreducible) algebraic $G$-variety $X$ is said to be \textit{wonderful of rank $r$} if
\smallbreak
\noindent
{\rm(i)}\enspace
it is smooth and complete,
\smallbreak
\noindent
{\rm(ii)}\enspace
it has an open $G$-orbit whose complement is the union of smooth prime $G$-divisors $D_i$ ($i=1,\dots,r$) with normal crossings 
and such that $\cap_1^r D_i\neq\emptyset$,
\smallbreak
\noindent
{\rm(iii)}\enspace
if $x,x'$ are such that $\{i:x\in D_i\}=\{i:x'\in D_i\}$ then $G.
x=G. x'$.
\smallbreak
\end{definition}

We shall say that a subgroup $H$ of $G$ (or $G/H$ itself) is \textit{wonderful} if the homogeneous space $G/H$ has a wonderful embedding, that is if $G/H$ can be realised as the dense $G$-orbit of a wonderful variety. This wonderful embedding is unique (up to isomorphism).

The radical of $G$ always acts trivially on a wonderful variety.

Note that the flag varieties are wonderful of rank $0$.
More generally, wonderful varieties are projective and spherical
(\textit{see} \cite{L96}).

The terminology \textit{spherical} means for an algebraic $G$-variety
that it is normal and has a dense $B$-orbit (e.g.\ normal toric varieties).
It follows that spherical varieties have finitely many $B$-orbits.

A subgroup $H$ of $G$ is called spherical if $G/H$ is spherical. Up to a change of representative in the conjugation class of $H$ in $G$, one may assume that $B\,H$ is open in $G$.

For a given spherical $G$-variety $X$ we will usually denote by $H$ its generic stabiliser, that is the stabiliser of a point in the dense $G$-orbit of $X$.

\bigbreak
\paragraph{\textbf{Colours.}}
The set $\Delta_X$ of \textit{colours} of $X$ is defined as the set of irreducible components of
the complementary  in $G/H$ of the dense $B$-orbit. Note that this
$B$-orbit is affine like every $B$-orbit, and so the colours are divisors of $G/H$. 

Let $P_X$ denote the stabiliser of the colours of $X$, it is a parabolic subgroup containing $B$.

\bigbreak
\paragraph{\textbf{Spherical roots.}}
Consider the field $\mathbb C(X)$ of rational functions on $X$
endowed with the dual action of $G$: 
$$
(g. f)(x)=f(g^{-1}. x) \quad\mbox{ for } f\in\mathbb C(X),\ g\in G \mbox{ and } x\in X.
$$
Denote by $\Xi_X$ the lattice formed by the weights of the $B$-eigenvectors in $\mathbb C(X)$.

Let $\mathcal V_X$ be the set of $G$-invariant $\mathbb Q$-valued discrete valuations of $\mathbb C(X)$;
where a $G$-invariant valuation $v$ is a valuation with the property that $v(g.f)=v(f)$ for any $g\in G$ and $f\in\mathbb C(X)$.
We will regard $\mathcal V_X$ as a subset of $\mathrm{Hom}(\Xi_X,\mathbb Q)$ via the injective map $\rho$ defined by:
$$
\rho(\nu)(\gamma)=\nu (f_\gamma)
$$ 
where $f_\gamma$ is a $B$-eigenvector of weight $\gamma$.
Note that $f_\gamma$ is uniquely determined up to scalar multiple.
Moreover, $\mathcal V_X$ is a simplicial cone in $\mathrm{Hom}(\Xi_X,\mathbb Q)$.
 
The set \textit{$\Sigma_X$ of spherical roots of $X$} is defined
as the set of primitive elements of $\Xi_X$ such that 
$$
\mathcal V_X=\{\chi\in\mathrm{Hom}(\Xi_X,\mathbb Q):\chi(\sigma)\leq 0,\mbox{ for all }\sigma\in\Sigma_X\}.
$$

Suppose $X$ is wonderful. 
Then the spherical roots of $X$ may also be defined as
the $T$-weights of $T_y X/T_y Y$ where $y$ is the point fixed by the Borel subgroup $B^-$ in $Y$, the (unique) closed $G$-orbit of $X$.
Further, the set $\Sigma_X$ is then a basis of the lattice $\Xi_X$.

Rank 1 wonderful varieties were classified by Ahiezer in \cite{Ah}.
As a consequence the possible spherical roots are all known; if $G$ is adjoint (i.e. of trivial centre), each spherical root is either a positive root or a sum of two such roots. In Section~\ref{Luna's diagrams}, we have reproduced from \cite{W} the list of all spherical roots, for $G$ adjoint, except the spherical root which is as well a simple root.

\begin{definition}
A wonderful variety $X$ is called \textit{strict} if each of its points has a selfnormalising stabiliser.
\end{definition}

This condition is equivalent to that given in Section~\ref{sphsys} in terms of colours and spherical roots. 
Pezzini proved that the strict wonderful varieties are exactly the wonderful varieties which have a simple immersion \textit{i.e.}\
those that can be embedded in the projectivisation of a simple $G$-module
(\textit{see} \cite{Pe07}).
Such varieties appear also in \cite{BCF} (\textit{see} also Section~\ref{uniqueness}).

Recall that in general 
the normaliser of a spherical subgroup is selfnormalising.
A selfnormalising spherical subgroup is wonderful (Corollary~7.6
in \cite{Kn96}); 
and a wonderful subgroup necessarily has finite index in its normaliser.

In the literature, important families of spherical varieties have
already appeared: symmetric varieties, spherical nilpotent orbits in
simple Lie algebras, model homogeneous spaces.
In the following, we shall gather some of their well-known properties.
Further details can be found in the corresponding appendices.

\subsection{Symmetric spaces}

Let $\sigma$ be a non-identical involution of $G$ and  $G^\sigma$ be the corresponding fixed point subset of $G$.
If $H$ is a subgroup of $G$ such that
$$
G^\sigma\subseteq H\subseteq N_G(G^\sigma)
$$
then the homogeneous space $G/H$ is called \textit{symmetric} or $H$ is called a \textit{symmetric subgroup} of $G$.
Here $N_G(G^\sigma)$ stands for the normaliser of $G^\sigma$ in $G$.

De Concini and Procesi proved in \cite{DP83} that
selfnormalising symmetric subgroups are wonderful.
More generally, symmetric subgroups are spherical  by \cite{Vu}.
Further $G^\sigma$ is connected (if $G$ is simply connected), reductive and of finite index in its normaliser (if $G$ is semisimple), \cite{St}.

The classification of involutions of $G$ was established by Cartan in the 1920es.
Vust described in \cite{Vu} the valuation cone and the colours associated to a symmetric space. 
We shall recall these results in Appendix~\ref{Symmetric varieties}.

\subsection{Spherical nilpotent orbits}
Let $G$ be semisimple.
Consider the adjoint action of $G$: $G$ acts by conjugation on its Lie algebra $\mathfrak g$.
Panyushev characterises in \cite{Pa94} (\textit{see} also \cite{Pa03} for a classification-free proof) the adjoint nilpotent orbits in $\mathfrak g$ which are $G$-spherical
and provides the list of spherical nilpotent orbits.
More specifically, he proves that an adjoint nilpotent orbit $G. e$ is spherical if and only if $(\mathrm{ad}\,e)^4=0$.

In Appendix~\ref{nilpotent orbits} we reproduce the list obtained in \cite{Pa94} of nilpotent spherical orbits of height 3, i.e.\ with $(\mathrm{ad}\,e)^3\neq0$.

\subsection{Model spaces}\label{modelspaces}
A quasi-affine homogeneous space $G/H$ is called \textit{model} if its
algebra of regular functions  decomposes as a $G$-module into the
multiplicity-free direct sum of all simple $G$-modules. In particular,
the subgroup $H$ is spherical (\textit{see} \cite{Bri97}).

Let us suppose $G$ to be semisimple.
The main theorem of \cite{Lu07} states that there exists a strict wonderful $G$-variety $X$, the \textit{model variety}, parametrising model $G$-homogeneous spaces explicitly as follows.

If the homogeneous space $G/H$ is model then $H$ stabilises a unique point in $X$ whose stabiliser equals $N_G(H)$.
And, conversely, for all $x\in X$ the homogeneous space $G/(G_x)^\sharp$ is model with $N_G((G_x)^\sharp)=G_x$. 
Here $(G_x)^\sharp$ stands for the intersection of the kernels of all characters of $G_x$. 

Luna obtained an explicit description of the \textit{principal} model homogeneous spaces $G/(G_x)^\sharp$, namely those with $x$ in the dense orbit of the model variety $X$, for any $G$. We report the generic stabiliser $G_x$ of $X$ in detail in Appendix~\ref{model spaces}, for $G$ simply connected.

\section{Spherical systems}
\label{sphsys}

Throughout this section, $G$ is assumed to be semisimple.

\begin{definition}
A \textit{strict spherical system} for $G$
is a couple
consisting of a subset $S^p$ of simple roots,
a set $\Sigma$ of spherical roots for $G$
(namely $T$-weights that are the spherical root of a rank 1 wonderful $G$-variety) 
which satisfies the following properties.
\smallbreak\noindent($\Sigma 1$)\enspace If $2\alpha \in \Sigma \cap 2S$ then $\frac{1}{2}\langle\alpha^\vee, \gamma \rangle$
is a non-positive integer for every $\gamma \in \Sigma \setminus \{ 2\alpha \}$.
\smallbreak\noindent($\Sigma 2$)\enspace If $\alpha, \beta \in S$ are orthogonal and $\alpha + \beta \in \Sigma$
then $\langle\alpha ^\vee,\gamma\rangle = \langle\beta ^\vee,\gamma\rangle$
for every $\gamma \in \Sigma$.
\smallbreak\noindent(S)\enspace For every $\gamma \in \Sigma$, there exists a rank 1 wonderful $G$-variety $X$
with $\gamma$ as spherical root and $S^p$ equal to the set of simple roots associated to $P_X$.
\smallbreak\noindent(R)\enspace For every $\gamma \in \Sigma$, there exists no rank 1 wonderful $G$-variety $X$ with $2\gamma$ as spherical root and $S^p$ equal to the set of simple roots associated to $P_X$. 
\end{definition}

It follows readily from the list of rank 1 wonderful varieties in \cite{W} 
that condition (R) implies that the set
$\Sigma$ of spherical roots can not contain simple roots, that is,
\smallbreak\noindent(R')\enspace $\Sigma\cap\ S=\emptyset$.

\begin{definition}
A spherical system for $G$ without simple spherical roots is a couple
$(S^p, \Sigma)$ as above satisfying conditions ($\Sigma 1$),
($\Sigma 2$), (S) and (R').
\end{definition}

In the following we will consider spherical systems without simple
spherical roots. These are slightly (but strictly) more general
objects than strict spherical systems.

In full generality, a spherical system is defined by Luna as a triple
($S^p$, $\Sigma$, $\mathbf A$) with $S^p$ and $\Sigma$ as above satisfying
conditions ($\Sigma 1$), ($\Sigma 2$) and (S).
The datum $\mathbf A$ is a multiset of functionals on $\Z\Sigma$ related to the simple roots
in $\Sigma$ with some extra conditions (\textit{see} Section~2.1
in \cite{Lu01} for details).

A spherical system without simple spherical roots is thus
a spherical system in Luna's sense of shape $(S^p, \Sigma, \emptyset)$.

\textit{By abuse of language, we shall say that a couple is a
spherical system whenever it is a spherical system without simple spherical roots.}

\subsection{Colours associated to a spherical system}

The set of colours $\Delta$ of a given spherical system $(S^p,\Sigma)$ is defined as
$$
\Delta=\left(S\setminus S^p\right)/\sim
$$
with $\alpha\sim\beta$ whenever $\alpha\perp\beta$ and $\alpha+\beta\in\Sigma$.

To avoid any confusion, the elements of $\Delta$ are denoted rather by $D_\alpha$ than by $\alpha$.
Hence, we have $D_\alpha=D_\beta$ if $\alpha\sim\beta$.

Define a $\mathbb Z$-linear map $\rho\colon \mathbb Z\Delta\rightarrow(\mathbb Z\Sigma)^\ast$ as follows, for all $\gamma\in\mathbb Z\Sigma$.
$$
\left\langle \rho(D_\alpha),\gamma\right\rangle =
\left \{
\begin{array} {ll}
\frac{1}{2}\left\langle\alpha^\vee,\gamma \right\rangle & \quad \mbox{if  $2\alpha\in\Sigma$}
 \\ 
\left\langle\alpha^\vee,\gamma \right\rangle & \quad \mbox{otherwise.} 
\end{array}
\right.
$$

\subsection{The spherical system of a wonderful variety}

Let $X$ be a wonderful $G$-variety and $\Delta_X$ be its set of colours, then $\mathrm{Pic}(X)\cong\mathbb Z\Delta_X$.
Let $P_X$ be the stabiliser of the colours of $X$.
The colours being $B$-stable, the subgroup $P_X$ is a parabolic subgroup containing $B$.
Recall that such parabolic subgroups are in correspondence with subsets of the set $S$ of simple roots.
Denote by $S^p_X$ the set of simple roots associated to $P_X$.

Suppose that the set $\Sigma_X$ of spherical roots of $X$ does not contain simple roots.

Define the map $\rho_X\colon\mathbb Z\Delta_X\to\Xi_X^\ast$ by
$\langle\rho_X(D),\gamma\rangle=v_D(f_\gamma)$, where $v_D$ is the
valuation associated to the divisor $D$ and $f_\gamma$ is the
$B$-eigenvector of weight $\gamma$ in $\mathbb C(X)$ (uniquely
determined up to a scalar). 

The pair $(S^p_X,\Sigma_X)$ is a spherical system
for $G$ and $\Delta_X$ is its  set of colours. 
It is usually referred as \textsl{the spherical system of $X$}; 
by analogy the spherical system of a wonderful homogeneous space means that of the corresponding wonderful embedding.
Further the map $\rho_X$ is the map $\rho$ defined in the previous section. 
\textit{See} Section~7 in \cite{Lu01} for details and for the more general case of an arbitrary wonderful variety.

\begin{remark}\label{rmk}
A wonderful variety without simple spherical roots with selfnormalising generic stabiliser is strict. This follows from the combinatorial characterisation of the spherical systems of the selfnormalising wonderful subgroups, \textit{see} for example \cite{Lo}, Theorem 2.
\end{remark}

\subsection{Luna's diagrams}\label{Luna's diagrams}
Let $G$ be of adjoint type. Following \cite{Lu01} (Section~4.1), one can assign to each spherical system a diagram built on the Dynkin diagram of $G$.
A colour $D=D_\alpha$ is represented by a circle drawn under (resp. around) the vertex $\alpha$
if $2\alpha\in\Sigma$ (resp. otherwise).
Whenever $D=D_\alpha=D_\beta$, we join the corresponding circles by a line.
The spherical roots are represented by shadowing some of the above circles or adding some zig-zag line or some number ``2'' as below; more specifically, we list in the following the rank 1 spherical systems with $\supp(\Sigma)=S$, as in \cite{W}, with their corresponding diagrams. 
For convenience, we also provide a labeling of these spherical systems inspired by that introduced in \cite{Lu01}.



\bigbreak\noindent{Type $\mathsf A$}

\rankonecase{a(n)} $n\geq2$, $S^p=S\setminus\{\alpha_1,\alpha_n\}$, $\Sigma=\{\alpha_1+\dots+\alpha_n\}$.
\[\begin{picture}(6000,600)\put(0,0){\usebox{\mediumam}}\end{picture}\]

\rankonecase{a'(1)}, $S^P=\emptyset$, $\Sigma=\{2\alpha_1\}$.
\[\begin{picture}(600,900)\put(0,0){\usebox{\aprime}}\end{picture}\]

\rankonecase{aa(1,1)}, $S^p=\emptyset$, $\Sigma=\{\alpha_1+\alpha'_1\}$.
\[\begin{picture}(4200,1050)\multiput(0,0)(3600,0){2}{\put(300,750){\circle*{150}}\put(300,750){\circle{600}}\put(300,0){\line(0,1){450}}}\put(300,0){\line(1,0){3600}}\end{picture}\]

\rankonecase{d(3)}, $S^p=\{\alpha_1,\alpha_3\}$, $\Sigma=\{\alpha_1+2\alpha_2+\alpha_3\}$.
\[\begin{picture}(3600,600)\multiput(0,300)(1800,0){2}{\usebox{\segm}}\put(1500,0){\usebox{\GreyCircle}}\end{picture}\]

\bigbreak\noindent{Type $\mathsf B$}

\rankonecase{b(n)} $n\geq2$, $S^p=S\setminus\{\alpha_1\}$, $\Sigma=\{\alpha_1+\dots+\alpha_n\}$.
\[\begin{picture}(7500,600)\put(0,0){\usebox{\shortbm}}\end{picture}\]

\rankonecase{b'(n)} $n\geq2$, $S^p=S\setminus\{\alpha_1\}$, $\Sigma=\{2\alpha_1+\dots+2\alpha_n\}$.
\[\begin{picture}(7500,1200)\put(0,0){\usebox{\shortbprimem}}\end{picture}\]

\rankonecase{b^\ast(n)} $n\geq2$, $S^p=S\setminus\{\alpha_1\alpha_2\}$, $\Sigma=\{\alpha_1+\dots+\alpha_n\}$.
\[\begin{picture}(7800,600)\put(0,0){\usebox{\shortbsecondm}}\end{picture}\]

\rankonecase{b^{\ast\ast}(3)}, $S^p=\{\alpha_1,\alpha_2\}$, $\Sigma=\{\alpha_1+2\alpha_2+3\alpha_3\}$.
\[\begin{picture}(3900,600)\put(0,0){\usebox{\bthirdthree}}\end{picture}\]

\bigbreak\noindent{Type $\mathsf C$}

\rankonecase{c(n)} $n\geq3$, $S^p=S\setminus\{\alpha_2\}$, $\Sigma=\{\alpha_1+2\alpha_2+\dots+2\alpha_i+\dots+2\alpha_{n-1}+\alpha_n\}$.
\[\begin{picture}(9000,600)\put(0,0){\usebox{\shortcm}}\end{picture}\]

\rankonecase{c^\ast(n)} $n\geq3$, $S^p=S\setminus\{\alpha_1,\alpha_2\}$, $\Sigma=\{\alpha_1+2\alpha_2+\dots+2\alpha_i+\dots+2\alpha_{n-1}+\alpha_n\}$.
\[\begin{picture}(9300,600)\put(0,0){\usebox{\shortcsecondm}}\end{picture}\]

\bigbreak\noindent{Type $\mathsf D$}

\rankonecase{d(n)} $n\geq4$, $S^p=S\setminus\{\alpha_1\}$, $\Sigma=\{2\alpha_1+\dots+2\alpha_i+\dots+2\alpha_{n-2}+\alpha_{n-1}+\alpha_n\}$.
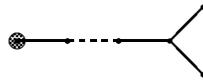
\[\begin{picture}(6900,2400)\put(0,0){\usebox{\shortdm}}\end{picture}\]

\bigbreak\noindent There are no spherical roots with support of type $\mathsf E$.

\bigbreak\noindent{Type $\mathsf F$}

\rankonecase{f(4)}, $S^p=\{\alpha_1,\alpha_2,\alpha_3\}$, $\Sigma=\{\alpha_1+2\alpha_2+3\alpha_3+2\alpha_4\}$.
\[\begin{picture}(5700,600)\put(0,0){\usebox{\ffour}}\end{picture}\]

\bigbreak\noindent{Type $\mathsf G$}

\rankonecase{g(2)}, $S^p=\{\alpha_2\}$, $\Sigma=\{2\alpha_1+\alpha_2\}$.
\[\begin{picture}(2100,600)\put(0,0){\usebox{\gtwo}}\end{picture}\]

\rankonecase{g'(2)}, $S^p=\{\alpha_2\}$, $\Sigma=\{4\alpha_1+2\alpha_2\}$.
\[\begin{picture}(2100,1200)\put(0,0){\usebox{\gprimetwo}}\end{picture}\]

\rankonecase{g^\ast(2)}, $S^p=\emptyset$, $\Sigma=\{\alpha_1+\alpha_2\}$.
\[\begin{picture}(2400,600)\put(0,0){\usebox{\gsecondtwo}}\end{picture}\]

\bigbreak\noindent For notational convenience, set $d(2)=aa(1,1)$, $b'(1)=a'(1)$, $c^\ast(2)=b^\ast(2)$.


\section{Combinatorial dictionary of wonderful varieties}\label{CombinatorialDictionary}

In the following we briefly recall the correspondence established by Luna
between geometrical features of a wonderful variety and some properties of its spherical system.
The statements are reported only in the setting of spherical systems without simple roots.
The proofs rely mainly on the theory of homogeneous embeddings
developed in \cite{LV}
in terms of combinatorial objects, the so-called coloured fans,  which generalise the notion of fan associated to a toric variety.
We shall not recall the proofs; \textit{see} Section~3 in \cite{Lu01} for details.

In the remainder, $\Delta$ is the set of colours of a given spherical system $(S^p,\Sigma)$. 

Further, we shall denote by $\supp \gamma$ the support of a spherical root $\gamma$ which
is defined as the set of simple roots $\alpha$ such that $\gamma=\sum n_\alpha \alpha$ with $n_\alpha\neq 0$.
More generally the support of a set of spherical roots is defined as the union of the supports of its elements.

\subsection{Localisation}

The \textit{localisation of $(S^p,\Sigma)$ at $S'$} for $S'\subset S$ is the spherical system $(S^p\cap S', \Sigma')$ where
$\Sigma'$ is the set of spherical roots in $\Sigma$ whose support is contained in $S'$.

Given a wonderful variety $X$ with spherical system $(S^p,\Sigma)$ and $S'$ a subset of $S$.
Consider the Levi subgroup $L=P_{S'}\cap P^-_{S'}$.
Let $C'$ be the connected centre of $L$ and $z$ the unique point of $X$ fixed by the Borel subgroup $B^-$.
The connected component $X^{S'}$ of the fixed point set $X^{C'}$ containing $z$ is a wonderful $L$-variety having
the localisation of $(S^p,\Sigma)$ at $S'$ as its spherical system.

\subsection{Quotient}

\smallbreak
A subset $\Delta'$ of $\Delta$ is said to be \textit{distinguished} 
if there exists $\phi\in\mathbb Z_{>0}\Delta'$
such that 
\begin{equation*}
\langle\rho(\phi),\gamma\rangle\geq 0\quad\mbox{ for all }\gamma\in\Sigma.
\end{equation*}

For a given distinguished subset $\Delta'$ of $\Delta$,
one defines:
\smallbreak
\noindent\enspace
$S^p/\Delta'=S^p\cup\{\alpha\in S\colon D_\alpha\in\Delta'\}$ 
\smallbreak
\noindent\enspace
$\Sigma/\Delta'$ as the set of indecomposable elements of the semigroup
given by the elements in $\mathbb Z_{\geq 0}\Sigma$
which are annihilated by $\rho(D)$ for each $D$ in $\Delta'$. 

If the couple ($S^p/\Delta'$, $\Sigma/\Delta'$) is a spherical system, it is called the \textit{quotient spherical system of $(S^p,\Sigma)$ by $\Delta'$}.

Given a wonderful variety $X$, let $(S^p,\Sigma)$ and $\Delta$ be respectively its spherical system and its set of colours.
If $\varphi\colon X\to X'$ is a surjective $G$-morphism  and $X'$ is a wonderful $G$-variety,
we shall denote by $\Delta_\varphi$ the set of colours of $X$ that are mapped dominantly onto $X'$.

\begin{proposition}[\cite{Lu01}, Proposition 3.3.2]\label{morphisms}
\smallbreak
\noindent{\rm(i)}\enspace
The set $\Delta_\varphi$ is a distinguished subset of $\Delta_X$.
Further, if the morphism $\varphi$ has connected fibers
then $X'$ has $(S^p/\Delta_\varphi,\Sigma/\Delta_\varphi)$ as spherical system.
\smallbreak
\noindent{\rm(ii)}\enspace
Conversely, let $\Delta'$ be a distinguished subset of $\Delta$ with quotient spherical system $(S^p/\Delta',\Sigma/\Delta')$.
Then there exist a unique (up to isomorphism) wonderful variety with spherical system $(S^p/\Delta',\Sigma/\Delta')$
as well as a unique surjective morphism $\varphi\colon X\rightarrow X'$ with connected fibers such that $\Delta_\varphi=\Delta'$.
\end{proposition}

A distinguished subset $\Delta'$ of $\Delta$ is said to be \textit{smooth} 
if $\Sigma/\Delta'$ is a subset of $\Sigma$.

\begin{proposition}[\cite{Lu01}, Proposition 3.3.3]
Let $\varphi\colon X\rightarrow X'$ be a morphism with connected fibers between two wonderful $G$-varieties.
The morphism $\varphi$ is smooth if and only if the subset $\Delta_\varphi$ is smooth.
\end{proposition}

\subsection{Parabolic induction}

Let $P$ be a parabolic subgroup of $G$. Suppose that $X$ is a wonderful $G$-variety and there exists a morphisms $X\rightarrow G/P$ whose fiber $Y$ over $P/P$ is acted on trivially by the radical of $P$. Then $Y$ is a wonderful variety (for a Levi subgroup of $P$) and $X$ is isomorphic to $G\times_P Y$, the fiber bundle over $G/P$ defined as the quotient of $G\times Y$ by the $P$-action
$$
p.(g,y)=(g.p^{-1}, p.y) \quad\mbox{ for  $p\in P$ and $(g,y)\in G\times Y$}.
$$

The variety $X$ is then said to be obtained by \textit{parabolic induction} from $Y$ by $P$. 

A subset $\Delta'$ of $\Delta$ is said to be \textit{homogeneous} if it is distinguished and such that $\Sigma/\Delta'=\emptyset$.
Denote by $\underline{\emptyset}(\Delta)$ the set given by the $D_\alpha$'s with $\alpha\in\supp(\Sigma)$.
Note that subsets of $\Delta$ containing $\underline{\emptyset}(\Delta)$ are homogeneous.

\begin{proposition}[\cite{Lu01}, Propositions 3.3.3 and 3.4]\label{P4} 
Given a wonderful $G$-variety $X$, let $(S^p,\Sigma)$ be its spherical system and $\Delta$ its set of colours. 

The homogeneous subsets of $\Delta$ are in one-to-one correspondence with the morphisms $\varphi\colon X\rightarrow G/P$ where $P$ is parabolic in $G$.
Further, the subsets of $\Delta$ containing $\underline{\emptyset}(\Delta)$ are in one-to-one correspondence
with the morphisms $\varphi$ where $X$ is obtained by  parabolic induction from $\varphi^{-1}(P/P)$ by $P$.

\end{proposition}

Therefore, a wonderful variety $X$ can be obtained by (non-trivial) parabolic induction if and only if $\supp(\Sigma_X)\not\perp (S\setminus\supp(\Sigma_X))$. 

A spherical system is said to be \textit{cuspidal} if $\supp(\Sigma)=S$.

\subsection{Decomposable spherical systems}

Let $\Delta_1$ and $\Delta_2$ be two distinguished non-empty subsets of $\Delta$. 
The set $\Delta_3=\Delta_1\cup \Delta_2$ is thus obviously distinguished.
The subsets $\Delta_1$ and $\Delta_2$ \textit{decompose} the spherical system $(S^p,\Sigma)$ if: 
\smallbreak\noindent\rm{(i)}\enspace 
$\Delta_1 \cap \Delta_2 = \emptyset$, 
\smallbreak\noindent\rm{(ii)}\enspace
$(\Sigma\setminus(\Sigma/\Delta_1)) \cap
(\Sigma\setminus(\Sigma/\Delta_2))=\emptyset$, 
\smallbreak\noindent\rm{(iii)}\enspace $((S^p/\Delta_1)\setminus S^p)\perp ((S^p/\Delta_2)\setminus S^p)$, 
\smallbreak\noindent\rm{(iv)}\enspace $\Delta_1$ or $\Delta_2$ is smooth.

\begin{proposition} [\cite{Lu01}, Proposition 3.5] 
Suppose $\Delta_1$ and $\Delta_2$ decompose a given spherical system $(S^p,\Sigma)$.
Assume also that there exists a wonderful $G$-variety $X_i$
with spherical system $(S^p/\Delta_i,\Sigma/\Delta_i)$ for $i=1,2,3$.
Then $X_1\times_{X_3}X_2$ is a wonderful $G$-variety with spherical system $(S^p,\Sigma)$.
\end{proposition}

\section{Classification of strict wonderful varieties}

\begin{theorem}\label{theorem}
Let $G$ be a semisimple group of adjoint type. To any spherical system ($S^p$,$\Sigma$) for $G$, where the set of spherical roots $\Sigma$ contains no simple root, there corresponds a (unique up to $G$-isomorphism) wonderful $G$-variety.
In particular, strict wonderful varieties are in bijective correspondence with strict spherical systems.
\end{theorem}

The above assertion was conjectured by Luna.
His conjecture is stated more generally for any wonderful variety and any spherical system (\textit{see} Appendix~A.2.2 in \cite{Lu07}), and can be suitably generalised to any semisimple group $G$ not necessarily of adjoint type. 
Losev recently proved the uniqueness part of the general conjecture (\cite{Lo}).
We propose an alternative approach in case of strict wonderful varieties (\textit{see} Section~\ref{uniqueness}).

A cuspidal and indecomposable spherical system is called \textit{primitive}.

By the results recalled in Section~\ref{CombinatorialDictionary}, it is enough to prove Theorem~\ref{theorem} for primitive spherical systems.

Primitive spherical systems (without simple spherical roots) and their corresponding wonderful subgroups are listed in Section~\ref{primitiveslist}. 

If $G$ is simply connected and $G/H$ is 
\smallbreak\noindent - \enspace
a symmetric space,
\smallbreak\noindent - \enspace
an adjoint nilpotent orbit of height 3 in $\mathfrak g$
\smallbreak\noindent - \enspace
or a principal model homogeneous space
\smallbreak\noindent
 then $G/H=\prod_i G_i/H_i$ and it appears that, for all $i$, the normaliser of $H_i$ is in the list of primitive cases.

These three families (are not disjoint and) do not cover all the primitive cases. 

There are primitive cases whose corresponding wonderful subgroup 
is the normaliser of some $H$ with $G/H$ a non-symmetric affine
homogeneous space, namely:
\smallbreak\noindent - \enspace $SL(n+1)/Sp(n)$ for $n$ even (this is also a model homogeneous space), 
\smallbreak\noindent - \enspace $Spin(2n+1)/GL(n)$ (here $G/N_G(H)$ is a model homogeneous space for $G$ adjoint), 
\smallbreak\noindent - \enspace $Spin(7)/G_2$, 
\smallbreak\noindent - \enspace $Spin(9)/Spin(7)$, 
\smallbreak\noindent - \enspace $SL(2)\times Sp(2n)/SL(2)\times Sp(2n-2)$, 
\smallbreak\noindent - \enspace $Sp(2n_1)\times Sp(2n_2)/SL(2)\times Sp(2n_1-2)\times Sp(2n_2-2)$, 
\smallbreak\noindent - \enspace $Sp(2n)/GL(1)\times Sp(2n-2)$, 
\smallbreak\noindent - \enspace $Spin(8)/G_2$, 
\smallbreak\noindent - \enspace $G_2/SL(3)$.
\smallbreak\noindent

Furthermore, in the primitive cases listed in
Section~\ref{primitiveslist} as 
\smallbreak\noindent 
(\ref{b''(n)}), (\ref{a(p)+b'(q)}), (\ref{acast(p)+b'(q)}) for $p$ even, (\ref{c''(n)}), (\ref{ca(1+q+1)}), (\ref{aa(1+p+1)+c''(q)}), (\ref{acast(p)+c''(q)}) for $q>2$, (\ref{ds(n)}), (\ref{a(p)+d(q)}), (\ref{acast(p)+d(q)}) for $p$ even, (\ref{fa(1+2+1)}) and  (\ref{fd(4)})
\smallbreak\noindent
the wonderful subgroup is selfnormalising but does not belong to any of the above families.

Finally, there are some primitive cases where the wonderful subgroup is not selfnormalising, namely 
\smallbreak\noindent
(\ref{b(n)}), (\ref{a(p)+b(q)}), (\ref{acast(p)+b(q)}), (\ref{cc(p+q)}) for $q=2$, (\ref{g(2)}).
\smallbreak\noindent
By Remark~\ref{rmk}, these are exactly the non-strict primitive spherical systems (without simple spherical roots).



\subsection{The list of primitive cases}\label{primitiveslist}

To simplify the notation, let us take $G$ simply connected. 
We report below spherical systems for $G/Z_G$, where $Z_G$ is the centre of $G$, hence $\Sigma$ will be included in the root lattice.
Therefore, the generic stabiliser $H$ of a wonderful variety with such a spherical system will always contain $Z_G$. 

After each spherical system we describe the corresponding subgroup $H$ of $G$. Naturally, $G$ is the simply connected group with the given Dynkin diagram.

\subsubsection*{Type $\mathsf A$}\begin{enumerate}

\item $\mathbf{aa(p,p)}$\label{aa(p,p)}, $p\geq1$,
  $S^p=\emptyset$, $\Sigma=\{\alpha_1+\alpha'_1,\dots$,
  $\alpha_i+\alpha'_i,\dots$, $\alpha_p+\alpha'_p\}$. 
$H=SL(p+1)\cdot Z_G$ (symmetric subgroup). 
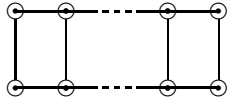
\[\diagramaapp\]

\item $\mathbf{ao(n)}$\label{ao(n)}, $n\geq1$, $S^p=\emptyset$,
  $\Sigma=\{2\alpha_1,\dots,2\alpha_i,\dots,2\alpha_n\}$. 
$H=SO(n+1)\cdot Z_G$ (symmetric subgroup).  
\[\diagramaon\]

\item $\mathbf{ac(n)}$\label{ac(n)}, $n\geq3$ odd, $S^p=\{\alpha_1,\dots,\alpha_{2i-1},\alpha_{2i+1},\dots,\alpha_n\}$, $\Sigma=$ $\{\alpha_1+2\alpha_2+\alpha_3,\dots$, $\alpha_{2i-1}+2\alpha_{2i}+\alpha_{2i+1},\dots$, $\alpha_{n-2}+2\alpha_{n-1}+\alpha_n\}$.  
$H=Sp(n+1)\cdot Z_G$ (symmetric subgroup).
\[\diagramacn\]

\item $\mathbf{aa(p+q+p)}$\label{aa(p+q+p)}, $n=2p+q$, $p\geq1$,
  $q\geq2$, $S^p=$ $\{\alpha_{p+2}, \dots$, $\alpha_{p+q-1}\}$, $\Sigma=\{\alpha_1+\alpha_n,\dots,\alpha_i+\alpha_{n+1-i},\dots,\alpha_p+\alpha_{p+q+1};\alpha_{p+1}+\dots+\alpha_{p+q}\}$. 
$H=(GL(p+q)\times GL(p+1))\cap G$ (symmetric subgroup).
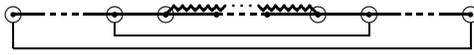
\[\diagramaapplusqplusp\]

\item $\mathbf{aa'(p+1+p)}$\label{aaast(p+1+p)}, $n=2p+1$, $p\geq1$, $S^p=\emptyset$, $\Sigma=\{\alpha_1+\alpha_n,\dots,\alpha_i+\alpha_{n+1-i},\dots,\alpha_p+\alpha_{p+2};2\alpha_{p+1}\}$. 
$H=N_G(SL(p+1)\times SL(p+1))$ (symmetric subgroup).
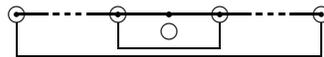
\[\diagramaaastpplusoneplusp\]

\item $\mathbf{a(n)}$\label{a(n)}, $n\geq2$, $S^p=\{\alpha_2,\dots,\alpha_{n-1}\}$, $\Sigma=\{\alpha_1+\dots+\alpha_n\}$. 
$H=GL(n)$ (symmetric subgroup).
\[\begin{picture}(6000,600)\put(0,0){\usebox{\mediumam}}\end{picture}\]

\item $\mathbf{ac^\ast(n)}$\label{acast(n)}, $n\geq3$, $S^p=\emptyset$, $\Sigma=\{\alpha_1+\alpha_2,\dots,\alpha_i+\alpha_{i+1},\dots,\alpha_{n-1}+\alpha_n\}$. \begin{itemize}
\item[-] $n$ odd: $H$ is the parabolic subgroup of semisimple type $\mathsf
  C_{(n-1)/2}$ of the symmetric subgroup $Sp(n+1)\cdot Z_G$.
\item[-] $n$ even: $H=Sp(n)\times GL(1)$. 
\end{itemize} 
\[\diagramacastn\]

\primitivetype{Type $\mathsf B$}

\item $\mathbf{bb(p,p)}$\label{bb(p,p)}, $p\geq2$, $S^p=\emptyset$, $\Sigma=\{\alpha_1+\alpha'_1,\dots$, $\alpha_i+\alpha'_i,\dots$, $\alpha_p+\alpha'_p\}$. 
$H=Spin(2p+1)\cdot Z_G$ (symmetric subgroup).
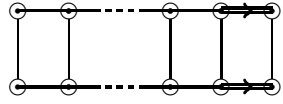
\[\diagrambbpp\]

\item $\mathbf{bo(p+q)}$\label{bo(p+q)}, $n=p+q$, $p\geq1$, $q\geq1$, $S^p=\{\alpha_{p+2},\dots,\alpha_n\}$, $\Sigma=\{2\alpha_1,\dots$, $2\alpha_i,\dots$, $2\alpha_p$; $2\alpha_{p+1}+\dots+2\alpha_n\}$. 
$H=Spin(p+1)\times Spin(2n-p)$ (symmetric subgroup).
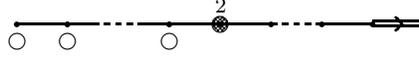
\[\diagrambopplusq\]

\item $\mathbf{b(n)}$\label{b(n)}, $n\geq2$, $S^p=S\setminus\{\alpha_1\}$, $\Sigma=\{\alpha_1+\dots+\alpha_n\}$. 
$H=Spin(2n)$ (symmetric subgroup).
\[\begin{picture}(7500,600)\put(0,0){\usebox{\shortbm}}\end{picture}\]

\item $\mathbf{b'(n)}$\label{b'(n)}, $n\geq2$,
  $S^p=S\setminus\{\alpha_1\}$,
  $\Sigma=\{2\alpha_1+\dots+2\alpha_n\}$. 
$H=N_G(Spin(2n))$ (symmetric subgroup).
\[\begin{picture}(7500,1200)\put(0,0){\usebox{\shortbprimem}}\end{picture}\]

\item $\mathbf{b^\ast(n)}$\label{b''(n)}, $n\geq2$,
  $S^p=\{\alpha_2,\dots,\alpha_{n-1}\}$,
  $\Sigma=\{\alpha_1+\dots+\alpha_n\}$. 
$H$ is the parabolic subgroup of semisimple type $\mathsf A_{n-1}$ of the symmetric subgroup $Spin(2n)$.
\[\begin{picture}(7500,600)\put(0,0){\usebox{\shortbsecondm}}\end{picture}\]

\item $\mathbf{bc^\ast(n)}$\label{bcast(n)}, $n\geq3$,
  $S^p=\emptyset$,
  $\Sigma=\{\alpha_1+\alpha_2,\dots,\alpha_i+\alpha_{i+1},\dots,\alpha_{n-1}+\alpha_n\}$.
\begin{itemize}
\item[-] $n$ odd: $H$ is the stabiliser of the line $[e]\in\mathbb
  P(\mathfrak g)$, where $e$ is a nilpotent element in the adjoint
  orbit of characteristic $(10\dots01)$.
\item[-] $n$ even: $H\subset P$, where $P$ is the parabolic subgroup of semisimple type $\mathsf A_{n-1}$ in the symmetric subgroup $Spin(2n)$, $H$ has the same radical and semisimple type $\mathsf C_{n/2}$.
\end{itemize}
\[\diagrambcastn\]

\item $\mathbf{bc'(n)}$\label{bc'(n)}, $n\geq2$, $S^p=\emptyset$, $\Sigma=\{\alpha_1+\alpha_2,\dots,\alpha_i+\alpha_{i+1},\dots,\alpha_{n-1}+\alpha_n;2\alpha_n\}$. $H=N_G(GL(n))$.
\[\diagrambcprimen\]

\item $\mathbf{a(p)+b(q)}$\label{a(p)+b(q)}, $n=p+q$, $p\geq2$,
  $q\geq2$, $S^p=$ $\{\alpha_2$, $\dots$, $\alpha_{p-1}$;
  $\alpha_{p+2}$, $\dots$, $\alpha_n\}$,
  $\Sigma=\{\alpha_1+\dots+\alpha_p,\alpha_{p+1}+\dots+\alpha_n\}$. 
$H$ is the parabolic subgroup of semisimple type $\mathsf A_{p-1}\times \mathsf D_{n-p}$ in the symmetric subgroup $Spin(2n)$, of index $2$ in its normaliser.
\[\diagramBsix\]

\item $\mathbf{a(p)+b'(q)}$\label{a(p)+b'(q)}, $n=p+q$, $p\geq2$,
  $q\geq1$, $S^p=$ $\{\alpha_2$, $\dots$, $\alpha_{p-1}$;
  $\alpha_{p+2}$, $\dots$, $\alpha_n\}$,
  $\Sigma=\{\alpha_1+\dots+\alpha_p,2\alpha_{p+1}+\dots+2\alpha_n\}$.
$H$ is the normaliser of the wonderful subgroup of case (\ref{a(p)+b(q)}).
\[\diagramBsixbis\]

\item $\mathbf{ac^\ast(p)+b(q)}$\label{acast(p)+b(q)}, $n=p+q$,
  $p\geq2$, $q\geq2$, $S^p=\{\alpha_{p+2},\dots,\alpha_n\}$,
  $\Sigma=\{\alpha_1+\alpha_2,\dots,\alpha_i+\alpha_{i+1},\dots,\alpha_{p-1}+\alpha_p;\alpha_{p+1}+\dots+\alpha_n\}$.
$H$ has index $2$ in its normaliser, the wonderful subgroup of case (\ref{acast(p)+b'(q)}).
\[\diagramacastpplusbq\]

\item $\mathbf{ac^\ast(p)+b'(q)}$\label{acast(p)+b'(q)}, $n=p+q$,
  $p\geq2$, $q\geq1$, $S^p=\{\alpha_{p+2},\dots,\alpha_n\}$,
  $\Sigma=\{\alpha_1+\alpha_2,\dots,\alpha_i+\alpha_{i+1},\dots,\alpha_{p-1}+\alpha_p;2\alpha_{p+1}+\dots+2\alpha_n\}$. 
\begin{itemize}
\item[-] $p$ even: $H\subset P$, where $P$ is the parabolic subgroup of semisimple type $\mathsf A_{p-1}\times\mathsf D_{n-p}$ in the symmetric subgroup $Spin(2n)$, $H$ has the same radical, semisimple type $\mathsf C_{p/2}\times\mathsf D_{n-p}$ and is selfnormalising.
\item[-] $p$ odd: $H$ is the stabiliser of the line $[e]\in\mathbb
  P(\mathfrak g)$, where $e$ is a nilpotent element in the adjoint
  orbit of characteristic $(10\dots010\dots0)$, with $\alpha_p(h)=1$.
\end{itemize}
\[\diagramacastpplusbprimeq\]

\item $\mathbf{b^{\ast\ast}(3)}$\label{b'''(3)},
  $S^p=\{\alpha_1,\alpha_2\}$,
  $\Sigma=\{\alpha_1+2\alpha_2+3\alpha_3\}$. $H=G_2\cdot Z_G$.
\[\begin{picture}(3900,600)\put(0,0){\usebox{\bthirdthree}}\end{picture}\]

\item $\mathbf{b^\ast(4)+b^{\ast\ast}(3)}$\label{b''(4)+b'''(3)}, 
  $S^p=\{\alpha_2,\alpha_3\}$,
  $\Sigma=\{\alpha_1+\alpha_2+\alpha_3+\alpha_4,\alpha_2+2\alpha_3+3\alpha_4\}$. $H=Spin(7)\cdot Z_G$.
\[\diagramBthree\]

\primitivetype{Type $\mathsf C$}

\item $\mathbf{cc(p,p)}$\label{cc(p,p)}, $p\geq3$,
  $S^p=\emptyset$, $\Sigma=\{\alpha_1+\alpha'_1,\dots$,
  $\alpha_i+\alpha'_i,\dots$, $\alpha_p+\alpha'_p\}$. 
$H=Sp(2p)\cdot Z_G$ (symmetric subgroup).
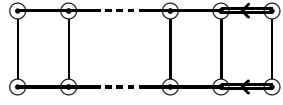
\[\diagramccpp\]

\item $\mathbf{co(n)}$\label{co(n)}, $n\geq3$, $S^p=\emptyset$,
  $\Sigma=\{2\alpha_1,\dots,2\alpha_i,\dots,2\alpha_n\}$. 
$H=N_G(GL(n))$ (symmetric subgroup).
\[\diagramcon\]

\item $\mathbf{c(n)}$\label{c(n)}, $n\geq3$,
  $S^p=S\setminus\{\alpha_2\}$,
  $\Sigma=\{\alpha_1+2\alpha_2+\dots+2\alpha_i+\dots+2\alpha_{n-1}+\alpha_n\}$.
$H=SL(2)\times Sp(2n-2)$ (symmetric subgroup).
\[\begin{picture}(9000,600)\put(0,0){\usebox{\shortcm}}\end{picture}\]

\item $\mathbf{cc(p+q)}$\label{cc(p+q)}, $n=p+q$, $p\geq2$ even,
  $q\geq2$, $S^p=\{\alpha_1,\dots$,
  $\alpha_{2i-1},\alpha_{2i+1},\dots$, $\alpha_{p+1}$;
  $\alpha_{p+3},\dots$, $\alpha_n\}$, $\Sigma=$
  $\{\alpha_1+2\alpha_2+\alpha_3,\dots$,
  $\alpha_{2i-1}+2\alpha_{2i}+\alpha_{2i+1},\dots$,
  $\alpha_{p-1}+2\alpha_p+\alpha_{p+1}$;
  $\alpha_{p+1}+2\alpha_{p+2}+\dots+2\alpha_{n-1}+\alpha_n\}$. 
$H=Sp(p+2)\times Sp(2n-p-2)$ (symmetric subgroup).
\[\diagramccpplusq\]

\item $\mathbf{cc'(p+2)}$\label{cc'(p+2)}, $n=p+2\geq4$ even, $S^p=$
  $\{\alpha_1$, $\dots$, $\alpha_{2i-1}$, $\alpha_{2i+1}$, $\dots$,
  $\alpha_{n-1}\}$, $\Sigma=$ $\{\alpha_1+2\alpha_2+\alpha_3,\dots$,
  $\alpha_{2i-1}+2\alpha_{2i}+\alpha_{2i+1},\dots$,
  $\alpha_{n-3}+2\alpha_{n-2}+\alpha_{n-1};2\alpha_{n-1}+2\alpha_n\}$. 
$H=N_G(Sp(n)\times Sp(n))$ (symmetric subgroup).
\[\diagramccprimepplustwo\]

\item $\mathbf{c^\ast(n)}$\label{c''(n)}, $n\geq3$,
  $S^p=\{\alpha_3,\dots,\alpha_n\}$,
  $\Sigma=\{\alpha_1+2\alpha_2+\dots+2\alpha_i+\dots+2\alpha_{n-1}+\alpha_n\}$.
$H$ is the parabolic subgroup of semisimple type $\mathsf C_{n-1}$ in the
symmetric subgroup $SL(2)\times Sp(2n-2)$.
\[\begin{picture}(9300,600)\put(0,0){\usebox{\shortcsecondm}}\end{picture}\]

\item $\mathbf{ca(1+q+1)}$\label{ca(1+q+1)}, $n=q+2$, $q\geq 2$,
  $S^p=\{\alpha_3,\dots,\alpha_q\}$,
  $\Sigma=\{\alpha_1+\alpha_n,\alpha_2+\dots+\alpha_{n-1}\}$.
$H$ is the parabolic subgroup of semisimple type $\mathsf A_{n-2}\times\mathsf A_1$ in the symmetric subgroup $SL(2)\times Sp(2n-2)$.
\[\diagramCfourprime\]

\item $\mathbf{aa(1+p+1)+c^\ast(q)}$\label{aa(1+p+1)+c''(q)}, $n=p+q+1$,
  $p\geq2$, $q\geq2$, $S^p=\{\alpha_3,\dots,\alpha_{p}$;
  $\alpha_{p+4},\dots,\alpha_n\}$,
  $\Sigma=\{\alpha_1+\alpha_{p+2},\alpha_2+\dots+\alpha_{p+1},\alpha_{p+2}+2\alpha_{p+3}+\dots+2\alpha_{n-1}+\alpha_n\}$. 
$H$ is the parabolic subgroup of semisimple type $\mathsf A_1\times \mathsf
A_p\times\mathsf C_{n-p-2}$ in the symmetric subgroup $SL(2)\times Sp(2n-2)$.
\[\diagramaaonepluspplusonepluscsecondq\]

\item $\mathbf{aa(1,1)+c^\ast(n)}$\label{aa(1,1)+c''(n_2)}, 
  $n\geq2$ (if $n=2$ the corresponding component is of type
  $\mathsf B_2$), $S^p=\{\alpha'_{3},\dots,\alpha'_{n}\}$,
  $\Sigma=\{\alpha_1+\alpha'_1,\alpha'_1+2\alpha'_2+\dots+2\alpha'_{n-1}+\alpha'_{n}\}$. 
$H=SL(2)\times Sp(2n-2)$.
\[\diagramCone\]

\item
  $\mathbf{aa(1,1)+c^\ast(n_1)+c^\ast(n_2)}$\label{aa(1,1)+c''(n_1)+c''(n_2)}, $n_1,n_2\geq2$ (if $n_i=2$ the corresponding component is of type $\mathsf B_2$), $S^p=$ $\{\alpha_3$, $\dots$, $\alpha_{n_1}$; $\alpha'_3$, $\dots$, $\alpha'_{n_2}\}$, $\Sigma=\{\alpha_1+\alpha'_1,\alpha_1+2\alpha_2+\dots+2\alpha_{n_1-1}+\alpha_{n_1},\alpha'_1+2\alpha'_2+\dots+2\alpha'_{n_2-1}+\alpha'_{n_2}\}$. 
$H=SL(2)\times Sp(2n_1-2)\times Sp(2n_2-2)$.
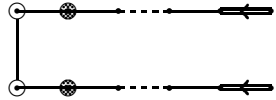
\[\diagramaaoneonepluscsecondnonepluscsecondntwo\]

\item $\mathbf{ac^\ast(p)+c^\ast(q)}$\label{acast(p)+c''(q)}, $n=p+q-1$,
  $p\geq2$, $q\geq2$, $S^p=\{\alpha_{p+2},\dots,\alpha_n\}$,
  $\Sigma=\{\alpha_1+\alpha_2,\dots,\alpha_i+\alpha_{i+1},\dots,\alpha_{p-1}+\alpha_p;\alpha_p+2\alpha_{p+1}+\dots+2\alpha_{n-1}+\alpha_n\}$. 
\begin{itemize}
\item[-] $p$ even: $H$ is the parabolic subgroup of semisimple type $\mathsf
  C_{p/2}\times\mathsf C_{n-1-p/2}$ in the symmetric subgroup $Sp(p)\times Sp(2n-p)$.
\item[-] $p$ odd: $H$ is the parabolic subgroup of semisimple type $\mathsf
  C_{(p-1)/2}\times \mathsf C_{n-(p+1)/2}$ in the symmetric subgroup
  $Sp(p+1)\times Sp(2n-p-1)$.
\end{itemize}
\[\diagramacastppluscsecondq\]

\item $\mathbf{a'(1)+c^\ast(q)}$\label{a'(1)+c''(q)}, $n=q\geq3$,
  $S^p=\{\alpha_3,\dots,\alpha_n\}$,
  $\Sigma=\{2\alpha_1,\alpha_1+2\alpha_2+\dots+2\alpha_{n-1}+\alpha_n\}$. 
$H=N_G(GL(1)\times Sp(2n-2))$.
\[\diagramCfive\]

\primitivetype{Type $\mathsf D$}

\item $\mathbf{dd(p,p)}$\label{dd(p,p)}, $p\geq4$,
  $S^p=\emptyset$, $\Sigma=\{\alpha_1+\alpha'_1,\dots$,
  $\alpha_i+\alpha'_i,\dots$,
  $\alpha_p+\alpha'_p\}$. $H=Spin(2p)\cdot Z_G$ (symmetric subgroup).
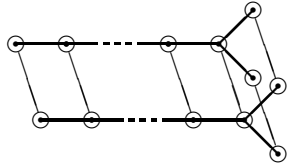
\[\diagramddpp\]

\item $\mathbf{do(p+q)}$\label{do(p+q)}, $n=p+q\geq4$, $p\geq1$,
  $q\geq2$, $S^p=\{\alpha_{p+2},\dots,\alpha_n\}$,
  $\Sigma=\{2\alpha_1,\dots,2\alpha_i,\dots$, $2\alpha_p$;
  $2\alpha_{p+1}+\dots+2\alpha_{n-2}+\alpha_{n-1}+\alpha_n\}$. 
$H=(Pin(p+1)\times Pin(2n-p-1))\cap G$ (symmetric subgroup). 
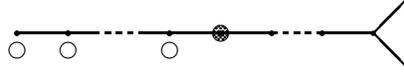
\[\diagramdopplusq\]

\item $\mathbf{do(n)}$\label{do(n)}, $n\geq4$, $S^p=\emptyset$,
  $\Sigma=\{2\alpha_1,\dots,2\alpha_i,\dots$, $2\alpha_n\}$. 
$H=Spin(n)\times Spin(n)$ (symmetric subgroup).
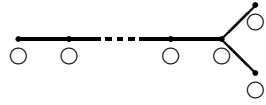
\[\diagramdon\]

\item $\mathbf{d(n)}$\label{d(n)}, $n\geq4$,
  $S^p=S\setminus\{\alpha_1\}$,
  $\Sigma=\{2\alpha_1+\dots+2\alpha_i+\dots+2\alpha_{n-2}+\alpha_{n-1}+\alpha_n\}$. 
$H=Spin(2n-1)\cdot Z_G$ (symmetric subgroup).
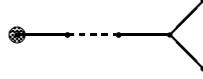
\[\begin{picture}(6900,2400)\put(0,0){\usebox{\shortdm}}\end{picture}\]

\item $\mathbf{dc'(n)}$\label{dc'(n)}, $n\geq6$ even,
  $S^p=\{\alpha_1,\dots,\alpha_{2i-1},\alpha_{2i+1},\dots,\alpha_{n-1}\}$, $\Sigma=$ $\{\alpha_1+2\alpha_2+\alpha_3,\dots$, $\alpha_{2i-1}+2\alpha_{2i}+\alpha_{2i+1},\dots$, $\alpha_{n-3}+2\alpha_{n-2}+\alpha_{n-1};2\alpha_n\}$. 
$H=N_G(GL(n))$ (symmetric subgroup).
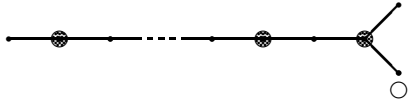
\[\diagramdcprimen\]

\item $\mathbf{dc(n)}$\label{dc(n)}, $n\geq5$ odd,
  $S^p=\{\alpha_1,\dots,\alpha_{2i-1},\alpha_{2i+1},\dots,\alpha_{n-2}\}$, $\Sigma=$ $\{\alpha_1+2\alpha_2+\alpha_3,\dots$, $\alpha_{2i-1}+2\alpha_{2i}+\alpha_{2i+1},\dots$, $\alpha_{n-4}+2\alpha_{n-3}+\alpha_{n-2};\alpha_{n-1}+\alpha_{n-2}+\alpha_n\}$. 
$H=GL(n)$ (symmetric subgroup).
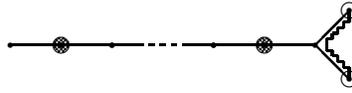
\[\diagramdcnodd\]

\item $\mathbf{ds(n)}$\label{ds(n)}, $n\geq4$,
  $S^p=\{\alpha_2,\dots,\alpha_{n-2}\}$,
  $\Sigma=\{\alpha_1+\dots+\alpha_{n-1},\alpha_1+\dots+\alpha_{n-2}+\alpha_n\}$.
$H$ is the parabolic subgroup of semisimple type $\mathsf A_{n-2}$ in the symmetric subgroup $Spin(2n-1)\cdot Z_G$.
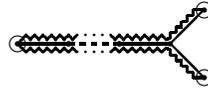
\[\diagramDseven\]

\item $\mathbf{ds^\ast(4)}$\label{dsast(4)},
  $S^p=\{\alpha_2\}$,
  $\Sigma=\{\alpha_1+\alpha_2+\alpha_3,\alpha_3+\alpha_2+\alpha_4,\alpha_4+\alpha_2+\alpha_1\}$. 
$H=K\cdot Z_G$, with $K$ simple subgroup of type $\mathsf G_2$.
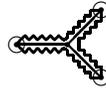
\[\diagramdsastfour\]

\item $\mathbf{dc^\ast(n)}$\label{dcast(n)}, $n\geq4$,
  $S^p=\emptyset$,
  $\Sigma=\{\alpha_1+\alpha_2,\dots,\alpha_i+\alpha_{i+1},\dots,\alpha_{n-2}+\alpha_n;\alpha_{n-2}+\alpha_n\}$. 
\begin{itemize}
\item[-] $n$ even: $H$ is the stabiliser of the line $[e]\in\mathbb
  P(\mathfrak g)$, where $e$ is a nilpotent element in the adjoint
  orbit of characteristic $(10\dots011)$.
\item[-] $n$ odd: $H\subset P$, where $P$ is the parabolic subgroup of
  semisimple type $\mathsf A_{n-2}$ in the symmetric subgroup
  $Spin(2n-1)\cdot Z_G$, $H$ has the same radical and semisimple type $\mathsf C_{(n-1)/2}$.
\end{itemize}
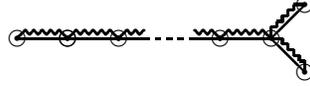
\[\diagramdcastn\]

\item $\mathbf{a(p)+d(q)}$\label{a(p)+d(q)}, $n=p+q$, $p,q\geq2$,
  $S^p=$ $\{\alpha_2$, $\dots$, $\alpha_{p-1}$; $\alpha_{p+2}$,
  $\dots$, $\alpha_n\}$,
  $\Sigma=\{\alpha_1+\dots+\alpha_p,2\alpha_{p+1}+\dots+2\alpha_{n-2}+\alpha_{n-1}+\alpha_n\}$. 
$H$ is the parabolic subgroup of semisimple type $\mathsf
A_{p-1}\times\mathsf B_{n-p-1}$ in the symmetric subgroup 
$Spin(2n-1)\cdot Z_G$.
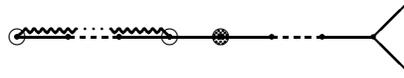
\[\diagramDfour\]

\item $\mathbf{ac^\ast(p)+d(q)}$\label{acast(p)+d(q)}, $n=p+q$,
  $p\geq2$, $q\geq2$, $S^p=\{\alpha_{p+2},\dots,\alpha_n\}$,
  $\Sigma=\{\alpha_1+\alpha_2,\dots,\alpha_i+\alpha_{i+1},\dots,\alpha_{p-1}+\alpha_p;2\alpha_{p+1}+\dots+2\alpha_{n-2}+\alpha_{n-1}+\alpha_n\}$. 
\begin{itemize}
\item[-] $p$ even: $H\subset P$, where $P$ is the parabolic subgroup of semisimple type $\mathsf A_{p-1}\times \mathsf B_{q-1}$ in the symmetric subgroup of type $Spin(2n-1)\cdot Z_G$, $H$ has the same radical and semisimple type $\mathsf C_{p/2}\times\mathsf B_{q-1}$.
\item[-] $p$ odd: $H$ is the stabiliser of the line $[e]\in\mathbb
  P(\mathfrak g)$, where $e$ is a nilpotent element in the adjoint
  orbit of characteristic $(10\dots010\dots0)$, with $\alpha_p(h)=1$.
\end{itemize}
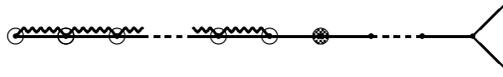
\[\diagramacastpplusdq\]

\primitivetype{Type $\mathsf E$}

\item $\mathbf{ee(p,p)}$\label{ee(p,p)}, $p=6,7,8$,
  $S^p=\emptyset$, $\Sigma=\{\alpha_1+\alpha'_1,\dots$,
  $\alpha_i+\alpha'_i,\dots$, $\alpha_p+\alpha'_p\}$. 
$H=K\cdot Z_G$ (symmetric subgroup) with $K$ simple of type $\mathsf E_p$.
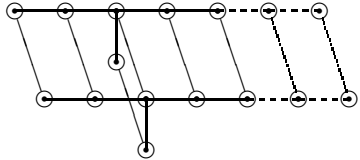
\[\diagrameepp\]

\item $\mathbf{eo(n)}$\label{eo(n)}, $n=6,7,8$, $S^p=\emptyset$,
  $\Sigma=\{2\alpha_1,\dots,2\alpha_i,\dots,2\alpha_n\}$. 
$H$ is the symmetric subgroup respectively equal to $Sp(8)\cdot Z_G$,
$SL(8)\cdot Z_G$, $Spin(16)$.
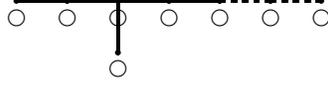
\[\diagrameon\]

\item $\mathbf{ea(6)}$\label{ea(6)}, $S^p=\emptyset$,
  $\Sigma=\{\alpha_1+\alpha_6,\alpha_3+\alpha_5,2\alpha_2,2\alpha_4\}$. 
$H=(SL(6)\times SL(2))\cdot Z_G$ (symmetric subgroup).
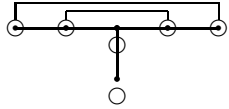
\[\diagrameasix\]

\item $\mathbf{ed(6)}$\label{ed(6)},
  $S^p=\{\alpha_3,\alpha_4,\alpha_5\}$,
  $\Sigma=\{\alpha_1+\alpha_3+\alpha_4+\alpha_5+\alpha_6,2\alpha_2+2\alpha_4+\alpha_3+\alpha_5\}$. 
$H=Spin(10)\times GL(1)$ (symmetric subgroup).
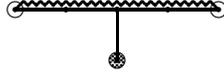
\[\diagramEtwo\]

\item $\mathbf{ef(6)}$\label{ef(6)},
  $S^p=\{\alpha_2,\alpha_3,\alpha_4,\alpha_5\}$,
  $\Sigma=\{2\alpha_1+2\alpha_3+2\alpha_4+\alpha_2+\alpha_5,\alpha_2+\alpha_3+2\alpha_4+2\alpha_5+2\alpha_6\}$. 
$H=K\cdot Z_G$ (symmetric subgroup) with $K$ simple of type $\mathsf F_4$.
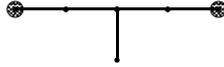
\[\diagramEone\]

\item $\mathbf{ec(7)}$\label{ec(7)},
  $S^p=\{\alpha_2,\alpha_5,\alpha_7\}$,
  $\Sigma=\{2\alpha_1,2\alpha_3,\alpha_2+2\alpha_4+\alpha_5,\alpha_5+2\alpha_6+\alpha_7\}$. 
$H=(Spin(12)\times SL(2))\cdot Z_G$ (symmetric subgroup).
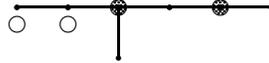
\[\diagramecseven\]

\item $\mathbf{ef(n)}$\label{ef(n)}, $n=7$ (or $8$),
  $S^p=\{\alpha_2,\alpha_3,\alpha_4,\alpha_5\}$,
  $\Sigma=\{2\alpha_1+2\alpha_3+2\alpha_4+\alpha_2+\alpha_5,2\alpha_6+2\alpha_5+2\alpha_4+\alpha_2+\alpha_3,2\alpha_7,(2\alpha_8)\}$. 
\begin{itemize}
\item[-] $n=7$: $H=N_G(K\times GL(1))$ (symmetric subgroup) with $K$
  simple of type $\mathsf E_6$.
\item[-] $n=8$: $H=(K\times SL(2))\cdot Z_G$ (symmetric subgroup) with
  $K$ simple of type $\mathsf E_7$.
\end{itemize}
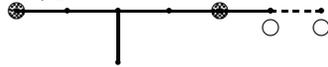
\[\diagramefprimen\]

\item $\mathbf{ec^\ast(n)}$\label{ecast(n)}, $n=6,7,8$,
  $S^p=\emptyset$,
  $\Sigma=\{\alpha_1+\alpha_3,\alpha_2+\alpha_4,\alpha_3+\alpha_4,\dots$, $\alpha_i+\alpha_{i+1},\dots$, $\alpha_{n-1}+\alpha_n\}$. 
\begin{itemize}
\item[-] $n=6$: $H$ is the parabolic subgroup of semisimple type $\mathsf C_3$ in the symmetric subgroup of type $\mathsf F_4$.
\item[-] $n=7$: $H$ is the stabiliser of the line $[e]\in\mathbb
  P(\mathfrak g)$, where $e$ is a nilpotent element in the adjoint
  orbit of characteristic $(0100001)$.
\item[-] $n=8$: $H$ is the stabiliser of $[e]\in\mathbb
  P(\mathfrak g)$, where $e$ is a nilpotent element with characteristic $(01000000)$.
\end{itemize}
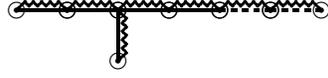
\[\diagramecastn\]

\item $\mathbf{ef(6)+a(2)}$\label{ef(6)+a(2)},
  $S^p=\{\alpha_2,\alpha_3,\alpha_4,\alpha_5\}$,
  $\Sigma=\{2\alpha_1+2\alpha_3+2\alpha_4+\alpha_2+\alpha_5,2\alpha_6+2\alpha_5+2\alpha_4+\alpha_2+\alpha_3,\alpha_7+\alpha_8\}$. 
$H$ is the stabiliser of the line $[e]\in\mathbb
  P(\mathfrak g)$, where $e$ is a nilpotent element in the adjoint
  orbit of characteristic $(00000010)$.
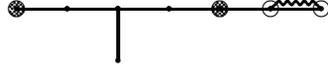
\[\diagramefsixplusatwo\]

\item $\mathbf{aa(2,2)+a(2)}$\label{aa(2,2)+a(2)},
  $S^p=\emptyset$,
  $\Sigma=\{\alpha_1+\alpha_6,\alpha_3+\alpha_5,\alpha_2+\alpha_4\}$. 
$H$ is the stabiliser of the line $[e]\in\mathbb
  P(\mathfrak g)$, where $e$ is a nilpotent element in the adjoint
  orbit of characteristic $(000100)$.
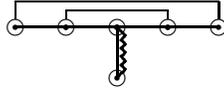
\[\diagramaatwotwoplusatwo\]

\item $\mathbf{ac(5)+a(2)}$\label{ac(5)+a(2)},
  $S^p=\{\alpha_2,\alpha_5,\alpha_7\}$,
  $\Sigma=\{\alpha_1+\alpha_3,\alpha_2+2\alpha_4+\alpha_5,\alpha_5+2\alpha_6+\alpha_7\}$. 
$H$ is the stabiliser of the line $[e]\in\mathbb
  P(\mathfrak g)$, where $e$ is a nilpotent element in the adjoint
  orbit of characteristic $(0010000)$.
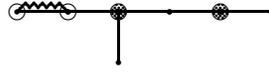
\[\diagramacfiveplusatwo\]

\primitivetype{Type $\mathsf F$}

\item $\mathbf{ff(4,4)}$\label{ff(4,4)}, $S^p=\emptyset$,
  $\Sigma=\{\alpha_1+\alpha'_1,\alpha_2+\alpha'_2,\alpha_3+\alpha'_3,\alpha_4+\alpha'_4\}$.
$H$ is a symmetric subgroup, simple of type $\mathsf F_4$.
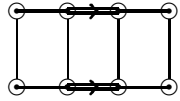
\[\diagramfffourfour\]

\item $\mathbf{fo(4)}$\label{fo(4)}, $S^p=\emptyset$,
  $\Sigma=\{2\alpha_1,2\alpha_2,2\alpha_3,2\alpha_4\}$. 
$H=Sp(6)\times SL(2)$ (symmetric subgroup).
\[\diagramfofour\]

\item $\mathbf{f(4)}$\label{f(4)}, 
  $S^p=S\setminus\{\alpha_4\}$,
  $\Sigma=\{\alpha_1+2\alpha_2+3\alpha_3+2\alpha_4\}$. 
$H=Spin(9)$ (symmetric subgroup).
\[\begin{picture}(5700,600)\put(0,0){\usebox{\ffour}}\end{picture}\]

\item $\mathbf{fa(1+2+1)}$\label{fa(1+2+1)}, $S^p=\emptyset$,
  $\Sigma=\{\alpha_1+\alpha_4,\alpha_2+\alpha_3\}$. $H$ is described
  in Section~\ref{exceptions}.
\[\diagramFone\]

\item $\mathbf{fd(4)}$\label{fd(4)}, $S^p=\{\alpha_2\}$,
  $\Sigma=\{\alpha_1+\alpha_2+\alpha_3,\alpha_2+2\alpha_3+\alpha_4\}$. 
$H$ is described
  in Section~\ref{exceptions}.
\[\diagramFthree\]

\item $\mathbf{ao(2)+a(2)}$\label{ao(2)+a(2)}, $S^p=\emptyset$,
  $\Sigma=\{\alpha_1+\alpha_2,2\alpha_3,2\alpha_4\}$. 
$H$ is the stabiliser of the line $[e]\in\mathbb
  P(\mathfrak g)$, where $e$ is a nilpotent element in the adjoint
  orbit of characteristic $(0100)$.
\[\diagramaotwoplusatwo\]

\item $\mathbf{fc^\ast(4)}$\label{fcast(4)}, $S^p=\emptyset$,
  $\Sigma=\{\alpha_1+\alpha_2,\alpha_2+\alpha_3,\alpha_3+\alpha_4\}$. 
$H$ is the parabolic subgroup of semisimple type $\mathsf A_1\times\mathsf
B_2$ in the symmetric subgroup $Spin(9)$.
\[\diagramfcastfour\]

\primitivetype{Type $\mathsf G$}

\item $\mathbf{gg(2,2)}$\label{gg(2,2)}, $S^p=\emptyset$,
  $\Sigma=\{\alpha_1+\alpha'_1,\alpha_2+\alpha'_2\}$. 
$H$ is a symmetric subgroup, simple of type $\mathsf G_2$.
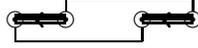
\[\diagramGtwo\]

\item $\mathbf{go(2)}$\label{go(2)}, $S^p=\emptyset$,
  $\Sigma=\{2\alpha_1,2\alpha_2\}$. 
$H=SL(2)\times SL(2)$ (symmetric subgroup).
\[\diagramGone\]

\item $\mathbf{g(2)}$\label{g(2)}, $S^p=\{\alpha_2\}$,
  $\Sigma=\{2\alpha_1+\alpha_2\}$. 
$H=SL(3)$.
\[\begin{picture}(2100,600)\put(0,0){\usebox{\gtwo}}\end{picture}\]

\item $\mathbf{g'(2)}$\label{g'(2)}, $S^p=\{\alpha_2\}$,
  $\Sigma=\{4\alpha_1+2\alpha_2\}$. 
$K=N_G(SL(3))$.
\[\begin{picture}(2100,1200)\put(0,0){\usebox{\gprimetwo}}\end{picture}\]

\item $\mathbf{g^\ast(2)}$\label{g''(2)}, $S^p=\emptyset$,
  $\Sigma=\{\alpha_1+\alpha_2\}$. 
$H$ is the stabiliser of the line $[e]\in\mathbb
  P(\mathfrak g)$, where $e$ is a nilpotent element in the adjoint
  orbit of characteristic $(10)$.
\[\begin{picture}(2400,600)\put(0,0){\usebox{\gsecondtwo}}\end{picture}\]

\end{enumerate}

\subsubsection{Two cases in type $\mathsf F$}\label{exceptions}
The wonderful subgroups $H$ of the rank two cases (\ref{fa(1+2+1)}) and
(\ref{fd(4)}) are known and can be described, like all the others, 
by providing 
a Levi factor $L$ and the unipotent radical $H^u$ as a $L$-representation
(\cite{W}).

Note that they are neither reductive, neither stabilisers of the line
spanned by a nilpotent element in the adjoint representation, nor
contained in a parabolic subgroup of a symmetric subgroup of $G$.

In order to get a more geometric description we realise them as
stabilisers of a line in a fundamental representation of $G$.

Let us denote by $\omega_i$, $i=1,2,3,4$, the $i$-th fundamental
weight of $G$ and by $V(\omega_i)$ the corresponding fundamental
representation.

\paragraph{\textit{Case \ref{fa(1+2+1)}}}
The subgroup $H$ is the stabiliser of a line $[v]\in\mathbb
P(V(\omega_2))$. The representation $V(\omega_1)$ is the adjoint
representation, its 2nd alternating power decomposes as
$\bigwedge^2V(\omega_1)=V(\omega_2)\oplus V(\omega_1)$. 
One can take the vector $v$ to be  
\[X_{\alpha_1+2\alpha_2+3\alpha_3+\alpha_4}\wedge
X_{\alpha_1+2\alpha_2+3\alpha_3+2\alpha_4}\in V(\omega_2),\]
where the $X_\gamma$'s are root vectors. It is of weight $\omega_3$.

\paragraph{\textit{Case \ref{fd(4)}}}
The subgroup $H$ is the stabiliser of a line $[v]\in\mathbb
P(V(\omega_3))$. 

The group $G$ consists of automorphisms of the exceptional Jordan
algebra $\mathcal J$ of hermitian $(3\times3)$-matrices of complex
Cayley octonions. 
The space of elements of trace 0 in $\mathcal J$ realises the
representation $V(\omega_4)$. Its 2nd alternating power decomposes as 
$V(\omega_3)\oplus V(\omega_1)$.

We omit the details, but it is worth saying that the investigated
vector $v\in V(\omega_3)$ can be explicitly found in the 
5-dimensional eigenspace of weight $\omega_4$.


\section{Proofs}

\subsection{Primitive spherical systems}
We shall explain how we determine all primitive spherical systems
without simple spherical roots. As in \cite{Lu01} (\textit{see} also
\cite{BP}), we make systematic use of the notion of
$\Delta$-connectedness.

\subsubsection{$\Delta$-Connectedness}
Take a spherical system $(S^p,\Sigma)$. 
Denote $\Delta$ its set of colours.

Let $S'$ be a subset of simple roots.
Define $\Delta(S')=\{D_\alpha:\alpha\in S'\}$.

Two spherical roots $\gamma_1,\gamma_2\in\Sigma$ are said to be
\textit{strongly $\Delta$-adjacent} if
\smallbreak
\noindent{\rm(i)}\enspace
$\langle\rho(D),\gamma_2\rangle\neq0$ for all
$D\in\Delta(\supp\gamma_1)$;
\smallbreak
\noindent{\rm(ii)}\enspace
$\langle\rho(D),\gamma_1\rangle\neq0$, for all
$D\in\Delta(\supp\gamma_2)$.

Given  a subset $\Sigma^\prime$ of $\Sigma$, the spherical system
$(S^p\cap\supp(\Sigma^\prime),\Sigma^\prime)$ is said to be
\textit{strongly $\Delta$-connected} if for every couple of spherical
roots $\gamma,\gamma'$ in $\Sigma^\prime$ there exists a sequence 
$$\gamma_1=\gamma,\gamma_2,\dots,\gamma_n=\gamma'$$
of spherical roots in $\Sigma'$, such that $\gamma_i$ is strongly
$\Delta$-adjacent to $\gamma_{i+1}$ for $i=1,\dots,n-1$. 
If $\Sigma^\prime\subset\Sigma$ is maximal with this property 
$(S^p\cap\supp(\Sigma^\prime),\Sigma^\prime)$ is called a
\textit{strongly $\Delta$-connected component} of $(S^p,\Sigma)$.
 
Denote $\Delta(\Sigma^\prime)$ the subset of colours
$D\in\Delta(\supp(\Sigma^\prime))$ such that 
$$\langle\rho(D),\gamma\rangle=0
\mbox{ for all } \gamma\in\Sigma\setminus\Sigma^\prime .
$$

We say that  $(S^p\cap\supp(\Sigma^\prime),\Sigma^\prime)$ is 
\smallbreak
\noindent
{\rm --}\enspace\textit{erasable} if there exists a nonempty smooth
distinguished subset of colours $\Delta^\prime$ included in
$\Delta(\Sigma^\prime)$.
\smallbreak
\noindent
{\rm --}\enspace\textit{quasi-erasable} if there exists a nonempty
distinguished subset of colours $\Delta^\prime$ included in
$\Delta(\Sigma^\prime)$ such that ($S^p/\Delta'$, $\Sigma/\Delta'$) is
a spherical system.

\begin{lemma}[\cite{BP} Lemma 3.9]\label{lemma:erasable} 
Given a spherical system $(S^p,\Sigma)$.
Let $\Sigma_1$ and $\Sigma_2$ be two disjoint subsets of $\Sigma$
giving two quasi-erasable localisations.
Suppose one of them (at least) is erasable.
Then the corresponding subsets
$\Delta^\prime_1\subset\Delta(\Sigma_1)$ and
$\Delta^\prime_2\subset\Delta(\Sigma_2)$ decompose the spherical
system $(S^p,\Sigma)$.
\end{lemma}

The strongly $\Delta$-connected component
$(S^p\cap\supp(\Sigma^\prime),\Sigma^\prime)$ is said to be
\textit{isolated} if the partition given by $\supp(\Sigma^\prime)$ and
its complement within $\supp(\Sigma)$ yields a decomposition of
$(S^p\cap\supp(\Sigma),\Sigma)$. An isolated component is erasable.

\subsubsection{General procedure}

To obtain the list of primitive spherical systems we first list
cuspidal strongly $\Delta$-connected spherical systems, by adding one
by one a strongly $\Delta$-adjacent (non-simple) spherical root. 
Then for each strongly $\Delta$-connected spherical system
$(S^p,\Sigma')$ one can realise how it can be a strongly
$\Delta$-connected component of a bigger spherical system: there will
be some colours $D$ such that $\rho(D)$ is non-zero on some spherical
root not belonging to $\Sigma'$, and it will be possible to check
whether such a component is isolated, erasable, quasi-erasable or none
of them. This property does not depend on the other components, but
only on the way of gluing. So one can easily realise that many
strongly $\Delta$-connected components are always erasable or
quasi-erasable no matter which is the whole spherical system. 
Therefore, to get all the remaining primitive spherical systems we
glue together, in all possible ways, two or more non-isolated cuspidal
strongly $\Delta$-connected spherical systems, avoiding pairs
satisfying the hypothesis of Lemma~\ref{lemma:erasable}.

\subsubsection{Cuspidal strongly $\Delta$-connected spherical systems}\ 

Isolated:
$aa(p,p)\ p\geq3$, $ao(n)\ n\geq3$, $ac(n)\ n\geq7$, $aa(p+q+p)\ p\geq2$, $aa'(p+1+p)$, $bb(p,p)$, $bo(p+q)$, $cc(p,p)$, $co(n)$, $c(n)$, $cc(p+q)$, $cc'(p+2)$, $ca(1+q+1)$, $dd(p,p)$, $do(p+q)$, $do(n)$, $dc'(n)$, $dc(n)\ n\geq7$, $ee(p,p)$, $eo(n)$, $ea(6)$, $ec(7)$, $ef(n)\ n=7,8$, $ec^\ast(n)\ n=7,8$, $ff(4,4)$, $fo(4)$, $f(4)$, $fa(1+2+1)$, $gg(2,2)$, $go(2)$, $g(2)$, $g'(2)$, $g^\ast(2)$.   

Erasable: 
$aa(1+q+1)$, $ac^\ast(n)$ $n\geq7$ odd, $bc^\ast(n)$ $n$ even, $dc(5)$, $dc^\ast(n)$ $n$ even, $ed(6)$, $b^{\ast\ast}(3)$.

Quasi-erasable: 
$ac^\ast(n)\ n\geq4$, $b^\ast(n)\ n\geq5$, $bc^\ast(n)\ n\geq5$,  $ds(n)$, $ds^\ast(4)$, $dc^\ast(n)$, $ec^\ast(6)$. 

Not necessarily quasi-erasable:
$aa(p,p)\ p=1,2$, $ao(n)\ n=1,2$, $ac(n)\ n=3,5$, $a(n)$, $ac^\ast(3)$, $b(n)$, $b'(n)$, $b^\ast(n)\ n=2,3,4$, $bc^\ast(3)$, $c^\ast(n)$, $d(n)$, $ef(6)$.

\subsubsection{A non-trivial example}

Consider the case $ac^\ast(n)$ with $n\geq3$. It is strongly
$\Delta$-connected. Let us call $\Sigma'$ its set of spherical
roots and $D_1,\dots,D_n$ its colours. Let us suppose it
corresponds to the strongly $\Delta$-connected component of a possibly
bigger spherical system.

If $\Delta(\Sigma')$ contains $D_2,\dots,D_{n-1}$, then it is itself
distinguished. If it contains $D_1$ (or $D_n$) also, then it is
smooth. If $n=3$, then $\{D_2\}$ is homogeneous.
If $\Delta(\Sigma')$ does not contain $D_2,\dots,D_{n-1}$, then the
underlying Dynkin diagram has a component of type $\mathsf D$ or
$\mathsf E$. In type $\mathsf D$ it can happen that
$\Delta(\Sigma')=\{D_1,D_3\}$, with $n=3$, and this is distinguished.
In type $\mathsf E$, it can happen that $D_2\notin\Delta(\Sigma')$,
for $n=3,4,5,6$, then $D_i\in\Delta(\Sigma')$ for $i\neq2,n$. 
If moreover $n\neq3$, the set of the $D_i$'s with $i$ odd 
is in $\Delta(\Sigma')$ and distinguished, 
since $D_n\in\Delta(\Sigma')$ or $n=4$. Otherwise, if
$n=3$ and $D_3\notin\Delta(\Sigma')$, then $\Delta(\Sigma')=\{D_1\}$
is not distinguished. However, in this last case 
$ac^\ast(3)$ can be a strongly $\Delta$-component 
of no primitive spherical system. 

\subsection{Uniqueness proof}
\label{uniqueness}
For a general connected reductive group $G$, Losev proves in \cite{Lo} that
two spherical $G$-homogeneous spaces $G/H_1$, $G/H_2$ that have the same
combinatorial invariants, $\Xi_{G/H_1}=\Xi_{G/H_2}$,
$\mathcal V_{G/H_1}=\mathcal V_{G/H_2}$ and $\Delta_{G/H_1}=\Delta_{G/H_2}$,
are $G$-equivariantly isomorphic. This implies that two (not
necessarily strict) wonderful $G$-varieties with the same spherical
system are $G$-equivariantly isomorphic.

For strict wonderful varieties, the same uniqueness result may be deduced from \cite{BCF} as we now explain.

First, let us recall the main results obtained in loc.\ cit.

Take $X$ a strict wonderful $G$-variety. Its spherical system (and more precisely its set of colours)
naturally provides an embedding of $X$ in the product of some projective spaces of
simple $G$-modules, say $V_i$ for $i=1,\dots,s$.
More specifically, $V_i$ is the irreducible $G$-module of highest weight given by the $B$-weight of a colour of $X$.
The variety $X$ being strict, these $B$-weights are linearly independent.

Let $\pi:\tilde X\dashrightarrow X$ be the affine multicone given by this embedding and $X_0$ be the affine multicone over the closed $G$-orbit of $X$. 
The $G$-variety $X_0$ is a spherical subvariety of $V$.

A spherical $G$-subvariety of $V$ is called \emph{an invariant deformation of $X_0$} if it is a (classical) deformation of $X_0$ and 
its coordinate ring is isomorphic as a $G$-module to the coordinate ring of $X_0$.

\begin{theorem}[\emph{see} Corollary 2.5 in~\cite{BCF}]
Every invariant deformation of $X_0$ is isomorphic to the closure of a $G$-orbit in $\tilde X$ on which $\pi$ is regular.
\end{theorem} 

Let $X$ and $X'$ be strict wonderful $G$-varieties.
Suppose they have the same spherical system. 
Let $x$ (resp. $x'$) be a point in $\tilde X$ (resp. $\tilde X'$) over the open $G$-orbit of $X$ (resp. of $X'$).
Note that $X$ and $X'$ have the same closed $G$-orbit since the latter is determined by the data $S^p$ of their common spherical system.
Then by the above theorem the stabilisers of $x$ and $x'$ in $G$ are conjugate.



The generic stabiliser of $X$ (and that of $X'$) equals the normaliser of the stabiliser of $x$ (and that of $x'$) in $G$; the uniqueness follows.

\subsection{Existence Proof}

We shall check case by case that the spherical subgroups $H$ of $G$
listed in Section~\ref{primitiveslist} are indeed wonderful and have
the spherical system which is stated. 
Recall that the cases of rank 1 and rank 2 are already known.

We shall make use of the following identities for a wonderful subgroup
$H$ of $G$:
\begin{lemma}[\cite{Lu01}, Section 5.2]\label{ids}
\begin{equation}\label{eq:dimGH}\dim G/H = \dim G/P_{S^p_{G/H}} + {\rm
    card}\,\Sigma_{G/H}; \end{equation}
\begin{equation}\label{eq:rankH}{\rm rank}\,\Xi(H) = {\rm card}\,
  \Delta_{G/H} - {\rm card}\, \Sigma_{G/H}. \end{equation}
\end{lemma}

The existence proof can be conducted as in \cite{Lu01,BP,Bra}.
Let us briefly recall the approach used there for proving that $G/H$ has a wonderful embedding and has the desired spherical system, where $H$ is the proposed wonderful subgroup.

One can first consider the case where $H$ is reductive (and spherical).
Such subgroups of semisimple groups are classified in \cite{Kr,Bri87} and ~\cite{Mi}.
When $H$ is properly included in its normaliser, we may need to
compute explicitly the colours in $G/H$ to deduce that $H$ is
wonderful. Indeed, by a result of Knop (Corollary 7.6 in~\cite{Kn96}),
the subgroup of the normaliser $N_G(H)$ which stabilises the set of the colours of $G/H$ is wonderful.
Once we know that $H$ is wonderful the identities in Lemma~\ref{ids} strongly
restrict the possibilities for the corresponding spherical system.

If $H$ is not reductive, to check that $H$ is spherical 
one can also use a criterion of Panyushev
(\textit{see} Corollary~1.4 in \cite{Pa94}).
Here one remarks  that $H$ is always (except in one case with 
rank greater than 2) selfnormalising hence wonderful. Notice that to compute the normaliser of $H$ it is enough to consider the normaliser of its reductive part. The computation of the latter can be done by hand.
From the identities stated in the lemma above, it is possible to
conclude with ad hoc arguments.

In this paper we rather proceed slightly differently, namely we shall consider
the wonderful varieties attached either to symmetric spaces (more
generally to affine spherical homogeneous spaces), to spherical
nilpotent orbits or to model spaces. We will then see that they
provide almost all wonderful varieties without simple spherical
roots.
The few remaining cases will be worked out separately.

\subsubsection{Symmetric spaces}
Recall the notation of Appendix~\ref{Symmetric varieties}.

\begin{proposition}~\label{sphericalsyst-symmetric}
In case B~II and C~II with $q=2$, there are two wonderful subgroups
without simple spherical roots, namely $G^\sigma$ and its normaliser.
The corresponding spherical systems are:  
\smallbreak\noindent
B~II:\enspace (\ref{b(n)}) $b(n)$ and (\ref{b'(n)}) $b'(n)$,
\smallbreak\noindent
C~II:\enspace (\ref{cc(p+q)}) $cc(p+q)$ with $q=2$ and (\ref{cc'(p+2)})
$cc'(p+2)$.

In all the remaining cases, there is one wonderful (selfnormalising)
symmetric subgroup without simple spherical roots. 
Further its spherical system is entirely determined by the basis $\tilde\Delta$ of its restricted
root system; the set of its spherical roots coincides with $\tilde\Delta$.

More specifically, the spherical systems of the selfnormalising
symmetric subgroups are (respectively to the list of simple
involutions given in the appendix):
\ref{aa(p,p)}, \ref{bb(p,p)}, \ref{cc(p,p)}, \ref{dd(p,p)},
\ref{ee(p,p)}, \ref{ff(4,4)}, \ref{gg(2,2)}; \ref{ao(n)}--\ref{a(n)},
\ref{bo(p+q)}, \ref{b'(n)}, 
\ref{co(n)}, \ref{c(n)}, \ref{cc(p+q)} for $q>2$,
\ref{cc'(p+2)}, \ref{do(p+q)}--\ref{dc(n)},
\ref{eo(n)}--\ref{ef(n)}, \ref{fo(4)}, \ref{f(4)},
\ref{go(2)}.
\end{proposition} 
     
\begin{proof}
Recall from Appendix~\ref{Symmetric varieties} that the spherical
roots of $G/H$ are given (up to a scalar) by the elements of its
restricted root system.

Consider first the cases distinct from case B~II and case C~II and set $H=N_G(G^\sigma)$ for the corresponding involutions $\sigma$. One sees that if $\tilde\alpha$ is an arbitrary element of one of these restricted root systems then
none of its multiples, except $\tilde \alpha$ itself, is a spherical root.
It follows that the set of spherical roots of $G/H$ does coincide with the corresponding restricted root system.
  
Let us work out now the case of the involutions labeled by B~II and
C~II with $q=2$.
Note first that $G^\sigma$ is of index $2$ in its normaliser.
Consider the following identity (Lemma~3.1 in \cite{Vu})
$$\Lambda\left(T_{-1}/(T_{-1}\cap G^\sigma)\right)=\left(\mathbb
  Z\tilde{\Phi}^\vee\right)^\ast .
$$
Together with the characterisation of the valuation cone of $G/G^\sigma$ recalled in the corresponding appendix,
we get that the $\alpha_i-\sigma(\alpha_i)$'s are indeed the spherical roots of $G/G^\sigma$.
Remark now that the $\alpha_i-\sigma(\alpha_i)$'s form a basis of $\left(\mathbb Z\tilde{\Phi}^\vee\right)^\ast $
hence of the character group $\Xi_{G/G^\sigma}=\Lambda\left(T_{-1}/(T_{-1}\cap G^\sigma)\right)$.
It follows that $G^\sigma$ is a wonderful subgroup of $G$.
\end{proof}

\subsubsection{Other affine homogeneous spaces}
First let us recall the following criterion of affinity; \textsl{see}~\cite{Bri97}.

A spherical homogeneous space $G/H$ is affine if and only if
there exists $\xi\in\mathbb Z_{\geq 0}\Sigma$ such that $ \left\langle
  \rho(\delta),\xi\right\rangle >0$ for every $\delta\in\Delta_{G/H}$.

As a consequence, we get

\begin{lemma}\label{affine-non-symmetric}
The spherical system of a non-symmetric affine wonderful homogeneous
space without 
simple spherical roots (provided it is cuspidal and indecomposable) is
one of the following:
\smallbreak\noindent
(\ref{acast(n)})
$ac^\ast(n)$ for $n$ even,
\smallbreak\noindent
(\ref{bc'(n)}) $bc'(n)$, (\ref{b'''(3)}) $b^{\ast\ast}(3)$, (\ref{b''(4)+b'''(3)}) $b^\ast(4)+b^{\ast\ast}(3)$,
\smallbreak\noindent
(\ref{aa(1,1)+c''(n_2)}) $aa(1,1)+c^\ast(n)$, (\ref{aa(1,1)+c''(n_1)+c''(n_2)}) $aa(1,1)+c^\ast(n_1)+c^\ast(n_2)$, (\ref{a'(1)+c''(q)}) $a'(1)+c^\ast(q)$,
\smallbreak\noindent
(\ref{dsast(4)}) $ds^\ast(4)$
\smallbreak\noindent
(\ref{g(2)}) $g(2)$ and (\ref{g'(2)}) $g'(2)$.
\end{lemma}

\begin{proposition}
Let $H\subset G$ be one of the following subgroups.
Then $G/N_G(H)$ has as its spherical system the corresponding entry in the list of Lemma~\ref{affine-non-symmetric}.
\smallbreak\noindent
{\rm -}\enspace \label{ahs1}$Sp(n)$ in $SL(n+1)$, for $n$ even.
\smallbreak\noindent {\rm -}\enspace\label{ahs2}$GL(n)$ in $Spin(2n+1)$, the connected stabiliser of a maximal isotropic subspace of $\mathbb C^{2n+1}$. 
\smallbreak\noindent
{\rm -}\enspace $G_2$ in $Spin(7)$, the connected subgroup of automorphisms of the complex Cayley octonions.
\smallbreak\noindent
{\rm -}\enspace$Spin(7)$ in $Spin(9)$.
\smallbreak\noindent
{\rm -}\enspace $SL(2)\times Sp(2n-2)$ in $SL(2)\times Sp(2n)$, where $H\subset SL(2)\times SL(2)\times Sp(2n-2)\subset G$, $SL(2)\times Sp(2n-2)$ being the symmetric subgroup of $Sp(2n)$ and $SL(2)$ diagonally embedded in $SL(2)\times SL(2)$.
\smallbreak\noindent
{\rm -}\enspace $SL(2)\times Sp(2n_1-2)\times Sp(2n_2-2)$ in $Sp(2n_1)\times Sp(2n_2)$, where $SL(2)\times Sp(2n_i-2)$ is the symmetric subgroup of $Sp(2n_i)$, $i=1,2$, and $SL(2)$ maps diagonally in $SL(2)\times SL(2)$.
\smallbreak\noindent
{\rm -}\enspace $GL(1)\times Sp(2n-2)$ in $Sp(2n)$.\label{ahs7}
\smallbreak\noindent
{\rm -}\enspace $G_2$ in $Spin(8)$, the connected subgroup of automorphisms of the complex Cayley octonions, as above $G_2\subset Spin(7)\subset Spin(8)$.
\smallbreak\noindent
{\rm -}\enspace \label{ahs9} $SL(3)$ in $G_2$, given by the root subsystem of long roots.
\end{proposition}

\begin{proof}
By \cite{Bri87} (and~\cite{Kr,Mi}), we know all connected spherical subgroups 
$H\subset G$ such that $G/H$ is affine: for $G$ simply connected, we have listed in the proposition all those which cannot be written as non-trivial product $G_1/H_1\times G_2/H_2$, are not symmetric and such that $G/N_G(H)$ has no simple spherical roots. The spherical system of $G/N_G(H)$ is thus cuspidal and satisfies the above criterion of affinity.

We already know rank 1 and rank 2 cases from \cite{W}.
The first and the second cases were worked out in~\cite{Lu07}; 
these are model spaces (\textit{see} Section~\ref{ss:model}).

To find all the corresponding spherical systems one can use the identities of Lemma \ref{ids}: they turn out to be indecomposable and thus in the list of Lemma \ref{affine-non-symmetric}.
\end{proof}

\begin{remark}
When $G$ is of type $\mathsf C_n$ or $\mathsf G_2$, the subgroup $H$
itself is also wonderful, but in type $\mathsf C_n$ the spherical
system of $H$ contains a simple spherical root (\textit{see} \cite{W}).
The corresponding case in type $\mathsf G_2$ is labeled as $(64)$ in the list of Section 5.1. 
\end{remark}

\subsubsection{Adjoint nilpotent orbits}

\begin{proposition}
The primitive spherical systems corresponding to a spherical nilpotent orbit are:
\smallbreak
\noindent\enspace
(\ref{bcast(n)}) $bc^\ast(n)$ for $n$ odd,
(\ref{acast(p)+b'(q)}) $ac^\ast(p)+b'(q)$ for $p$ odd,
\smallbreak\noindent\enspace
(\ref{dcast(n)}) $dc^\ast(n)$ for $n$ even,
(\ref{acast(p)+d(q)}) $ac^\ast(p)+d(q)$ for $p$ odd,
\smallbreak\noindent\enspace
(\ref{ecast(n)}) $ec^\ast(n)$ for $n=7,8$,
(\ref{ef(6)+a(2)}) $ef(6)+a(2)$,
\smallbreak\noindent\enspace
(\ref{aa(2,2)+a(2)}) $aa(2,2)+a(2)$,
(\ref{ac(5)+a(2)}) $ac(5)+a(2)$,
\smallbreak\noindent\enspace
(\ref{ao(2)+a(2)}) $ao(2)+a(2)$, 
(\ref{g''(2)}) $g^\ast(2)$.
\end{proposition}

\begin{proof}
We will keep the notation of Appendix~\ref{nilpotent orbits} 
and as recalled there, 
we shall be concerned here only by nilpotent elements of height 3. 
One can compute the spherical system 
of the (non-necessarily primitive) symmetric space $L/N_L(K)$ 
(by means for example of Proposition~\ref{sphericalsyst-symmetric}); 
note that the Lie algebras $\mathfrak l$ and $\mathfrak k$ 
are given in the appendix.
And in turn, one gets the spherical system of $G/N_G(H)P^r$ 
by parabolic induction from that of $L/N_L(K)\simeq P/N_G(H)P^r$.
Here $P^r$ denotes the radical of the parabolic subgroup $P$.

Moreover, there exists a dominant morphism with connected fibers 
between the wonderful embeddings of $G/N_G(H)$ and of $G/N_G(H)P^r$ 
(\textit{see} Proposition~\ref{morphisms}).
We then check case by case that there exists only one spherical system 
admitting the spherical system of $N_G(H)P^r$ as quotient 
and satisfying the identity in (\ref{eq:dimGH}). 
This gives the spherical system of $N_G(H)$. 
The proposition follows.
\end{proof}

\begin{lemma}
Let $e\in\mathfrak g$ be a nilpotent element and $[e]$ be the line spanned by $e$. If $H=G_e$ is spherical then $N_G(H)/H$ is one-dimensional and $N_G(H)=G_{[e]}$.
\end{lemma}

\begin{proof}
Recall the notation of Appendix~\ref{nilpotent orbits}.
The subspace $\mathbb C h$ is the maximal reductive subalgebra contained in $\mathfrak n_\mathfrak g(e)/\mathfrak g_e$  where
$\mathfrak n_\mathfrak g(e)$ is the Lie algebra of the normaliser $N_G(H)$; \emph{see} for example \cite{BK}.
If $H$ is spherical then  $N_G(H)/H$ is a diagonalisable group; \emph{see} for example \cite{Bri97}.
The lemma follows readily.
\end{proof}

\begin{remark}
Due to the previous lemma, we have listed in Section~\ref{primitiveslist} the wonderful subgroups corresponding to the above spherical systems as stabilisers of the line $[e]$ spanned by a nilpotent element $e$.
\end{remark}

\begin{remark}
The fact that the above spherical systems
do not contain simple spherical roots implies the saturation of the
weight cone of spherical adjoint nilpotent orbits (Conjecture~5.12 in
\cite{Pa03}, Section~1.3 in \cite{Lu07}).
\end{remark}

\subsubsection{Model homogeneous spaces}\label{ss:model}
Spherical systems of model spaces have already been computed by Luna.

\begin{proposition}(\cite{Lu07})
The spherical system of a model $G$-variety $X$ does not depend on the
isogeny type of $G$ except in type $\mathsf B$.
\smallbreak
\noindent{\rm(i)}\enspace
If $G$ is adjoint of type $\mathsf B$, the spherical system is
(\ref{bc'(n)}) $bc'(n)$. 
\smallbreak
\noindent{\rm(ii)}\enspace
In the simply connected case, the spherical roots of $X$ can all be
written as the sum of two non-orthogonal simple roots.
Their spherical systems are: (\ref{acast(n)}) $ac^\ast(n)$,
(\ref{bcast(n)}) $bc^\ast(n)$, (\ref{acast(p)+c''(q)})
$ac^\ast(p)+c^\ast(q)$ for $q=2$, (\ref{dcast(n)}) $dc^\ast(n)$,
(\ref{ecast(n)}) $ec^\ast(n)$, (\ref{fcast(4)}) $fc^\ast(4)$ and
(\ref{g''(2)}) $g^\ast(2)$.
\end{proposition}

Note that the cases 
(\ref{bcast(n)}) with $n$ odd, (\ref{dcast(n)}) with $n$ even,
(\ref{ecast(n)}) and (\ref{g''(2)}) 
belong also to the family of nilpotent orbits whereas the case
(\ref{acast(n)}) with $n$ even and (\ref{bc'(n)}) are affine.

\subsubsection{A few remaining cases.}
From the list of primitive spherical systems there are a few cases not included in the above families.

Most of them are of rank 1 or 2: 
\smallbreak\noindent\enspace
(\ref{b''(n)}) $b^\ast(n)$, (\ref{a(p)+b(q)}) $a(p)+b(q)$, (\ref{a(p)+b'(q)}) $a(p)+b'(q)$,
\smallbreak\noindent\enspace
(\ref{c''(n)}) $c^\ast(n)$, (\ref{ca(1+q+1)}) $ca(1+q+1)$, (\ref{ds(n)}) $ds(n)$, (\ref{a(p)+d(q)}) $a(p)+d(q)$,
\smallbreak\noindent\enspace
 (\ref{fa(1+2+1)}) $fa(1+2+1)$, (\ref{fd(4)}) $fd(4)$.

The respective wonderful subgroups are already described in \cite{W}
and reported in the list of primitive cases.

The remaining spherical systems are:
\smallbreak\noindent\enspace (\ref{acast(p)+b(q)}) $ac^\ast(p)+b(q)$,
\smallbreak\noindent\enspace (\ref{acast(p)+b'(q)}) $ac^\ast(p)+b'(q)$ for $p$ even,
\smallbreak\noindent\enspace (\ref{aa(1+p+1)+c''(q)}) $aa(1+p+1)+c^\ast(q)$ 
\smallbreak\noindent\enspace (\ref{acast(p)+c''(q)}) $ac^\ast(p)+c^\ast(q)$ for $q>2$,
\smallbreak\noindent\enspace (\ref{acast(p)+d(q)}) $ac^\ast(p)+d(q)$ for $p$ even.

Those in type $\mathsf B$ and $\mathsf D$ are very similar to some cases arising from spherical adjoint nilpotent orbits of height 3. The two cases in type $\mathsf C$ have already been treated in \cite{Pe03}.

The cases (\ref{acast(p)+b'(q)}) and (\ref{acast(p)+d(q)}) are easily handled by localisation from the $p$ odd analogues.

\smallbreak
\paragraph{\textbf{Case \ref{acast(p)+b(q)}}}
Consider the wonderful subgroup $H$ with spherical system
$ac^\ast(p)+b'(q)$ and let $H^\circ$ be the identity connected
component of $H$. Then $H^\circ$ is the investigated wonderful subgroup. 
Indeed, note that  $H^\circ$ has index $2$ in $H$ and $H=N_G(H^\circ)$. 
To prove that $H^\circ$ is wonderful one needs to compute explicitly
the colours and the spherical roots. In comparison with $H$, one sees
easily that the colours are represented by the same regular $B\times
H$-semiinvariant functions on $G$ and that the spherical roots are the
same except for the last one which is just divided by 2. The subgroup
$H^\circ$ is then wonderful with spherical system equal to
$ac^\ast(p)+b(q)$.

\smallbreak
\paragraph{\textbf{Case \ref{aa(1+p+1)+c''(q)}.}}
Let $G=Sp(2n)$.
Take $H=K\cdot H^u$  with $K=GL(p+1)\times SL(2)\times Sp(2n-2p-4)$
and the Lie algebra of $H^u$, as a $K'$-module, being isomorphic to
the direct sum $V(\omega_1+\omega_1'')\oplus V(2\omega_1)$. The
subgroup $H$ is spherical, selfnormalising hence wonderful and it is
not a parabolic induction. We have: $H\subset H_1=KQ^u\subset Q$ where
$Q$ is the standard parabolic of $G$ associated to $\alpha_{p+1}$.
Note that $H_1$ is wonderful in $G$:
$G/H_1$ is a parabolic induction of the symmetric space $L/K$, that is
C II.  
One thus checks that there is only one spherical system having such a
quotient: it is $aa(1+p+1)+c^\ast(q)$.

\smallbreak
\paragraph{\textbf{Case \ref{acast(p)+c''(q)}.}}
This is a generalisation of the model case in type $\mathsf C_n$ of Paragraph~\ref{ss:model}.
Here $G=SO_ {2n}$.
Let  $K_1$ be the symmetric group of $G$ of type $\mathsf C_{p/2}\times \mathsf C_{n-p/2}$
if $p$ is even
and of type $\mathsf C_{(p+1)/2}\times \mathsf C_{n-(p+1)/2}$ if $p$ is odd.
Take $H$ to be the parabolic subgroup of semisimple type $\mathsf C_{p/2}\times \mathsf C_{n-1-p/2}$ if $p$ is even
and of type $\mathsf C_{(p-1)/2}\times \mathsf C_{n-(p+1)/2}$ if $p$ is odd.
In all cases the subgroup $H$ is spherical and selfnormalising hence wonderful in $G$.
It follows that the spherical system of $G/K_1$ is a quotient of that of $G/H$.
Note that the spherical system of $G/K_1$ is $cc(p'+q')$, where $p'=(p/2)-1$ if $p$ is even, 
and $p'=(p-1)/2$ if $p$ is odd.
Moreover, $H=K\cdot H^u$ is included in $H_1=K\cdot Q^u$,
where $Q$ is the standard parabolic subgroup of $G$ of semisimple type $\mathsf C_{n-1}$.
The subgroup $H_1$ is wonderful in $G$; its spherical system is a parabolic induction of $cc(p'+(q'-1))$.
The spherical system $ac^\ast(p)+c^\ast(q)$ is the only one, with the right dimension, admitting  the spherical systems of $G/K_1$ and $G/H_1$ as quotients.

\appendix

\section{Symmetric spaces}
\label{Symmetric varieties}

Cartan's classification of involutions of semisimple groups appears
usually in the literature in terms of Satake diagrams or Kac diagrams
(\textit{see} for instance Chapter~X in \cite{He78}).
Analogously to these diagrams, the so-called restricted root systems determine (up to conjugation) the involutions.
Hence Cartan's classification can be also given in terms of restricted root systems.

Before reproducing this classification, we shall recall
the definition of a restricted root system associated to an involution
and how this root system was related by Vust to the sets of spherical roots of the corresponding symmetric spaces (recall that a symmetric space is spherical).

For an expository on symmetric varieties, one may also refer to~\cite{Ti}.

\subsection{Restricted root system (after~\cite{Vu})}

Given a non-identical involution $\sigma$ of a semisimple group $G$,
let $T_{-1}$ be a maximal $\sigma$-anisotropic torus, namely
a torus on which $\sigma$ acts as the inversion and
which is maximal for this property.
Take a maximal torus $T$ containing $T_{-1}$. Then $T$ is $\sigma$-stable;
the lattice $\Lambda$ of $T$-weights inherits an involution $\sigma$
and the root system  $\Phi$ attached to $T$ is $\sigma$-stable.

Then the set $\tilde\Phi$ of non-zero elements
$\alpha-\sigma(\alpha)$, for $\alpha\in\Phi$, is a root system for
$\Lambda(T_{-1})\otimes\mathbb R$ which may not be reduced.
It is called \textit{the restricted root system associated to $\sigma$}. 

One may choose a Borel subgroup $B$ such that:
if $\alpha$ is a  positive root relative to $B$ then $\sigma(\alpha)$
is either $\alpha$ or it is a negative root. 
Let $H$ be a symmetric group corresponding to $\sigma$ \textit{i.e.}\
$G^\sigma \subset H\subset N_G(G^\sigma)$.
Then $BH$ is open in $G$ and the set
$$
\tilde{\Delta}
=\{\alpha_i-\sigma(\alpha_i):\alpha_i\in\Delta\}\setminus \{0\}
$$
is a basis of $\tilde\Phi$.

Further the lattice $\Xi_{G/H}$ of $B$-weights of $\mathbb C(G/H)$ can
be identified to the character group of $T_{-1}/(T_{-1}\cap H)$ and
the cone of $G$-invariant valuations $\mathcal V(G/H)$ is the
antidominant Weyl chamber of the dual root system $\tilde{\Phi}^\vee$.
Hence the elements of $\tilde{\Delta}$ are equal (up to a scalar) to
the spherical roots of $G/H$ (Proposition~2 in Section~2.4 of \cite{Vu}).

\subsection{Cartan's classification}
Every symmetric space (for a connected reductive group) is a product of a torus, symmetric spaces $G_1\times G_1/G_1$ (with $G_1$ simple and diagonally embedded) and symmetric spaces for simple groups.
 
If $G=G_1\times G_1$ with $G_1$ simple then $\tilde\Phi=\langle\alpha_1+\alpha_1',\dots,\alpha_n+\alpha_n'\rangle$.

The classification for $G$ simple is presented in the following list.
It is reproduced from~\cite{He78} but giving the restricted root system instead of the Satake diagram.
The restricted roots systems have been computed case by case by means of the recalls made above; \emph{see} also~\cite{Vu} where some cases are worked out in details. 
More precisely, we give the Cartan label, the restricted root system  and the Lie algebra $\mathfrak h$ of the fixed point subgroup. 
The Cartan label specifies the type of $G$. The semisimple rank of $G$ is denoted by $n$.

\begin{itemize}
\item[A I:] $\tilde\Phi=\langle 2\alpha_1,\dots,2\alpha_n \rangle$ of type $\mathsf A_n$, $\mathfrak h=\mathfrak{so}(n+1)$.
\item[A II:] $n$ odd ($n\geq3$), $\tilde\Phi=\langle \alpha_1+2\alpha_2+\alpha_3,\dots,\alpha_{n-2}+2\alpha_{n-1}+\alpha_n \rangle$ of type $\mathsf {A}_{(n-1)/2}$, $\mathfrak h=\mathfrak{sp}(n+1)$.
\item[A III:] $n=2p+q$ ($p,q\geq1$),
\begin{itemize} 
\item[] if $q\geq2$, $\tilde\Phi=\langle \alpha_1+\alpha_{n},\dots,\alpha_p+\alpha_{n-p+1},\alpha_{p+1}+\dots+\alpha_{p+q} \rangle$ of type $\mathsf {BC}_{p+1}$, $\mathfrak h=\mathfrak{sl}(p+1)+\mathfrak{sl}(p+q)+\mathfrak{gl}(1)$;
\item[] if $q=1$, $\tilde\Phi=\langle \alpha_1+\alpha_{n},\dots,\alpha_p+\alpha_{n-p+1},2\alpha_{p+1} \rangle$ of type $\mathsf {C}_{p+1}$, $\mathfrak h=\mathfrak{sl}(p+1)+\mathfrak{sl}(p+1)+\mathfrak{gl}(1)$.
\end{itemize}
\item[A IV:]
\begin{itemize} 
\item[] if $n\geq2$, $\tilde\Phi=\langle \alpha_1+\dots+\alpha_n \rangle$ of type $\mathsf {A}_1$, $\mathfrak h=\mathfrak{gl}(n)$;
\item[] if $n=1$, $\tilde\Phi=\langle 2\alpha_1 \rangle$ of type $\mathsf {A}_1$, $\mathfrak h=\mathfrak{gl}(1)$.
\end{itemize}
\item[B I:] $n=p+q$ ($p,q\geq1$), $\tilde\Phi=\langle 2\alpha_1,\dots,2\alpha_p,2\alpha_{p+1}+\dots+2\alpha_n \rangle$ of type $\mathsf {B}_{p+1}$, $\mathfrak h=\mathfrak{so}(p+1)+\mathfrak{so}(2n-p)$.
\item[B II:] $\tilde\Phi=\langle 2\alpha_1+\dots+2\alpha_n \rangle$ of type $\mathsf {A}_1$, $\mathfrak h=\mathfrak{so}(2n)$.
\item[C I:] $\tilde\Phi=\langle 2\alpha_1,\dots,2\alpha_n \rangle$ of type $\mathsf {C}_n$, $\mathfrak h=\mathfrak{gl}(n)$.
\item[C II:] $n=p+q$ ($p\geq0$ even, $q\geq2$), 
\begin{itemize}
\item[] if $q\geq3$, $\tilde\Phi=\langle \alpha_1+2\alpha_2+\alpha_3,\dots,\alpha_{p-1}+2\alpha_p+\alpha_{p+1},\alpha_{p+1}+2\alpha_{p+2}+\dots+2\alpha_{n-1}+\alpha_n \rangle$ of type $\mathsf {BC}_{(p/2)+1}$, $\mathfrak h=\mathfrak{sp}(p+2)+\mathfrak{sp}(2n-p-2)$;
\item[] if $q=2$, $\tilde\Phi=\langle \alpha_1+2\alpha_2+\alpha_3,\dots,\alpha_{n-3}+2\alpha_{n-2}+\alpha_{n-1},2\alpha_{n-1}+2\alpha_n \rangle$ of type $\mathsf {C}_{(p/2)+1}$, $\mathfrak h=\mathfrak{sp}(n)+\mathfrak{sp}(n)$.
\end{itemize}
\item[D I:] $n=p+q$ ($p\geq1$, $q\neq1$), 
\begin{itemize}
\item[] if $q\geq2$, $\tilde\Phi=\langle 2\alpha_1,\dots,2\alpha_p,2\alpha_{p+1}+\dots+2\alpha_{n-2}+\alpha_{n-1}+\alpha_n \rangle$ of type $\mathsf {B}_{p+1}$, $\mathfrak h=\mathfrak{so}(p+1)+\mathfrak{so}(2n-p-1)$;
\item[] if $q=0$, $\tilde\Phi=\langle 2\alpha_1,\dots,2\alpha_n \rangle$ of type $\mathsf {D}_n$, $\mathfrak h=\mathfrak{so}(n)+\mathfrak{so}(n)$;
\end{itemize}
\item[D II:] $\tilde\Phi=\langle 2\alpha_1+\dots+2\alpha_{n-2}+\alpha_{n-1}+\alpha_n \rangle$ of type $\mathsf {A}_1$, $\mathfrak h=\mathfrak{so}(2n-1)$.
\item[D III:] 
\begin{itemize}
\item[] if $n$ is even, $\tilde\Phi=\langle \alpha_1+2\alpha_2+\alpha_3,\dots,\alpha_{n-3}+2\alpha_{n-2}+\alpha_{n-1},2\alpha_n \rangle$ of type $\mathsf {C}_{n/2}$, $\mathfrak h=\mathfrak{gl}(n)$;
\item[] if $n$ is odd, $\tilde\Phi=$ $\langle
  \alpha_1+2\alpha_2+\alpha_3,\dots$,
  $\alpha_{n-4}+2\alpha_{n-3}+\alpha_{n-2}$, $\alpha_{n-2}+\alpha_{n-1}+\alpha_n \rangle$ of type $\mathsf {BC}_{(n-1)/2}$, $\mathfrak h=\mathfrak{gl}(n)$.
\end{itemize}
\item[E I:] $n=6$, $\tilde\Phi=\langle 2\alpha_1,\dots,2\alpha_6 \rangle$ of type $\mathsf {E}_6$, $\mathfrak h=\mathfrak{sp}(8)$.
\item[E II:] $n=6$, $\tilde\Phi=\langle \alpha_1+\alpha_6,\alpha_3+\alpha_5,2\alpha_2,2\alpha_4 \rangle$ of type $\mathsf {F}_4$, $\mathfrak h=\mathfrak{sl}(6)+\mathfrak{sl}(2)$.
\item[E III:] $n=6$, $\tilde\Phi=\langle \alpha_1+\alpha_3+\alpha_4+\alpha_5+\alpha_6,2\alpha_2+2\alpha_4+\alpha_3+\alpha_5 \rangle$ of type $\mathsf {BC}_2$, $\mathfrak h=\mathfrak{so}(10)+\mathfrak{gl}(1)$.
\item[E IV:] $n=6$, $\tilde\Phi=\langle 2\alpha_1+2\alpha_3+2\alpha_4+\alpha_2+\alpha_5,\alpha_2+\alpha_3+2\alpha_4+2\alpha_5+2\alpha_6 \rangle$ of type $\mathsf {A}_2$, $\mathfrak h=\mathfrak{f}_4$.
\item[E V:] $n=7$, $\tilde\Phi=\langle 2\alpha_1,\dots,2\alpha_n \rangle$ of type $\mathsf {E}_7$, $\mathfrak h=\mathfrak{sl}(8)$.
\item[E VI:] $n=7$, $\tilde\Phi=\langle 2\alpha_1,2\alpha_3,\alpha_2+2\alpha_4+\alpha_5,\alpha_5+2\alpha_6+\alpha_7 \rangle$ of type $\mathsf {F}_4$, $\mathfrak h=\mathfrak{so}(12)+\mathfrak{sl}(2)$.
\item[E VII:] $n=7$, $\tilde\Phi=\langle 2\alpha_1+2\alpha_3+2\alpha_4+\alpha_2+\alpha_5,\alpha_2+\alpha_3+2\alpha_4+2\alpha_5+2\alpha_6,2\alpha_7 \rangle$ of type $\mathsf {C}_3$, $\mathfrak h=\mathfrak{e}_6+\mathfrak{gl}(1)$.
\item[E VIII:] $n=8$, $\tilde\Phi=\langle 2\alpha_1,\dots,2\alpha_n \rangle$ of type $\mathsf {E}_8$, $\mathfrak h=\mathfrak{so}(16)$.
\item[E IX:] $n=8$, $\tilde\Phi=\langle 2\alpha_1+2\alpha_3+2\alpha_4+\alpha_2+\alpha_5,\alpha_2+\alpha_3+2\alpha_4+2\alpha_5+2\alpha_6,2\alpha_7,2\alpha_8 \rangle$ of type $\mathsf {F}_4$, $\mathfrak h=\mathfrak{e}_7+\mathfrak{sl}(2)$.
\item[F I:] $\tilde\Phi=\langle 2\alpha_1,2\alpha_2,2\alpha_3,2\alpha_4 \rangle$ of type $\mathsf {F}_4$, $\mathfrak h=\mathfrak{sp}(6)+\mathfrak{sl}(2)$.
\item[F II:] $\tilde\Phi=\langle \alpha_1+2\alpha_2+3\alpha_3+2\alpha_4 \rangle$ of type $\mathsf {BC}_1$, $\mathfrak h=\mathfrak{so}(9)$.
\item[G:] $\tilde\Phi=\langle 2\alpha_1,2\alpha_2 \rangle$ of type $\mathsf {G}_2$, $\mathfrak h=\mathfrak{sl}(2)+\mathfrak{sl}(2)$.
\end{itemize}

\section{Spherical nilpotent orbits}
\label{nilpotent orbits}
We collect some results on adjoint nilpotent orbits from \cite{Pa94, Pa03} unless otherwise stated.

\subsection{Main properties}
Take a nilpotent element $e$ in the Lie algebra $\mathfrak g$ of a
simple group $G$.
By the Jacobson-Morozov theorem, there exist $h,f\in\mathfrak g$ such
that $(e,h,f)$ is a $SL(2)$-triple, this means that
$$
[h,e]=2e,\quad [h,f]=-2f\quad\mbox{ and }\quad[e,f]=h.
$$

The semisimple element $h$ yields a $\mathbb Z$-grading on $\mathfrak g$:
$\mathfrak g=\oplus_{i\in\mathbb Z}\mathfrak g(i)$
where $\mathfrak g(i)=\{x\in\mathfrak g:[h,x]=ix\}.$

Set
$$\mathfrak p=\oplus_{i\geq0}\mathfrak g(i),\quad\mathfrak l=\mathfrak g(0)\quad\mbox{ and }\quad \mathfrak n_1=\oplus_{i\geq1}\mathfrak g(i).
$$
Let $\mathfrak g_e$ be the centraliser of $e$ in $\mathfrak g$. It is included in $\mathfrak p$.
Set
$$\mathfrak k=\mathfrak g_e\cap \mathfrak g(0)\quad\mbox{ and }\quad \mathfrak n=\mathfrak g_e\cap\oplus_{i\geq1}\mathfrak g(i).$$

Denote by $P$ the connected subgroup of $G$ (resp. $L$, $P^u$, $H$, $H^u$ and $K$)
with Lie algebra $\mathfrak p$ (resp. $\mathfrak l$, $\mathfrak n_1$, $\mathfrak g_e$, $\mathfrak n$ and $\mathfrak k$).
Then $P=P^u L$ (resp. $H=H^uK$) is a Levi decomposition of $P$ (resp. $H$). 
Further, 
$\dim H=\dim\mathfrak g(0) +\dim\mathfrak g(1)$ and $ \dim H^u=\dim\mathfrak g(1) +\dim\mathfrak g(2)$.

The linear map $\mathrm{ad}\,f\colon\oplus_{i\geq 3}\mathfrak g(i)\to\mathfrak n_1/\mathfrak n$ is an isomorphism of $\mathfrak k$-modules. 

Define the \textit{height} of $e$ to be the minimal positive integer such that $(\mathrm{ad}\, e)^m\neq 0$.
The adjoint orbit $G. e$ is spherical if and only if the height of $e$ is 2 or 3.

If $G/H$ is spherical then $K$ is the fixed point subgroup of an involution in the reductive group $L$. 

Assume that the height of $e$ is equal to 2. Then $\mathfrak g(i)=0$ for all $i\geq 3$ and $H^u=P^u$.
In other words, the spherical homogeneous space $G/H$ is obtained by parabolic induction from $L/K$.
The wonderful subgroup $N_G(H)$ is thus obtained by parabolic induction from a symmetric subgroup of the derived group $(L,L)$.

\subsection{Panyushev's classification}
\label{Panyushev}
Choose now a Cartan subalgebra $\mathfrak t$ and a Borel subalgebra $\mathfrak b$ 
such that $h\in\mathfrak t\subseteq\mathfrak l$ and $\mathfrak b\subseteq\mathfrak p$.
One then gets a set of simple roots $\alpha_i$ such that $0\leq\alpha_i(h)\leq 2$ for all $i$.
The sequence of the non-negative integers $\alpha_i(h)$ is called the \textit{characteristic} of the orbit;
different orbits have different characteristics.
Further, the subalgebra $\mathfrak l$ is thus the standard Levi subalgebra corresponding to the subset of simple roots with $\alpha_i(h)=0$.

We report below the classification of spherical nilpotent orbits of height 3 (the height 2 case being not cuspidal as just explained)
for $G$ simple.
We shall give the type of the simple group $G$, 
the characteristic of the nilpotent orbit,
the subalgebra $\mathfrak k$ as well as the $\mathfrak k$-module $\mathfrak n_1/\mathfrak n$ (which can be computed as in \cite{Lu01,BP}).
In case of classical type, the nilpotent orbit can be described by the size of its Jordan blocks 
and we shall also provide the corresponding partition.

\smallbreak
\noindent{\rm -}
\enspace
Type $\mathsf B_{2r+1}$, $r\geq1$, $(10\dots01)$:
$\mathfrak k\cong \mathfrak{sp}(2r)$ and $\mathfrak n_1/\mathfrak n\cong V(\omega_1)$; partition $(3,2^{2r})$.

\smallbreak
\noindent{\rm -}
\enspace
Type $\mathsf B_{2r+s+1}$, $r,s\geq1$, $(10\dots010\dots0)$ with $\alpha_{2r+1}(h)=1$:
$\mathfrak k\cong \mathfrak{sp}(2r)\oplus \mathfrak{so}(2s)$ and $\mathfrak n_1/\mathfrak n\cong V(\omega_1)$;
partition $(3,2^{2r},1^{2s})$.

\smallbreak
\noindent{\rm -}\enspace
Type $\mathsf D_{2r+2}$, $r\geq1$: $(10\dots 011)$;
$\mathfrak k\cong \mathfrak{sp}(2r)$ and $\mathfrak n_1/\mathfrak n\cong V(\omega_1)$; partition $(3,2^{2r},1)$.

\smallbreak
\noindent{\rm -}
\enspace
Type $\mathsf D_{2r+s+2}$, $r,s\geq1$, $(10\dots010\dots0)$ with $\alpha_{2r+1}(h)=1$:
$\mathfrak k\cong \mathfrak{sp}(2r)\oplus \mathfrak{so}(2s+1)$ and
$\mathfrak n_1/\mathfrak n\cong V(\omega_1)$; partition $(3,2^{2r},1^{2s+1})$.

\smallbreak
\noindent{\rm -}\enspace
Type $\mathsf E_6$: $(000100)$; $\mathfrak k\cong \mathfrak{sl}(3)\oplus\mathfrak{sl}(2)$ and $\mathfrak n_1/\mathfrak n\cong V(\omega_1')$.

\smallbreak
\noindent{\rm -}\enspace
Type $\mathsf E_7$: $(0010000)$; $\mathfrak k\cong \mathfrak{sl}(2)\oplus\mathfrak{sp}(6)$ and $\mathfrak n_1/\mathfrak n\cong V(\omega_1)$.

\smallbreak
\noindent{\rm -}\enspace
Type $\mathsf E_7$: $(0100001)$; $\mathfrak k\cong \mathfrak{sp}(6)$ and $\mathfrak n_1/\mathfrak n\cong V(\omega_1)$.

\smallbreak
\noindent{\rm -}\enspace
Type $\mathsf E_8$: $(00000010)$; $\mathfrak k\cong \mathfrak{f}_4\oplus\mathfrak{sl}(2)$ and $\mathfrak n_1/\mathfrak n\cong V(\omega_1')$.

\smallbreak
\noindent{\rm -}\enspace
Type $\mathsf E_8$: $(01000000)$; $\mathfrak k\cong \mathfrak{sp}(8)$ and $\mathfrak n_1/\mathfrak n\cong V(\omega_1)$.

\smallbreak
\noindent{\rm -}\enspace
Type $\mathsf F_4$: $(0100)$, $\mathfrak k\cong \mathfrak{sl}(2)\oplus\mathfrak{so}(3)$ and $\mathfrak n_1/\mathfrak n\cong V(\omega_1)$.

\smallbreak
\noindent{\rm -}\enspace
Type $\mathsf G_2$: $(10)$; $\mathfrak k\cong \mathfrak{sl}(2)$ and $\mathfrak n_1/\mathfrak n\cong V(\omega_1)$.

\section{Luna's classification of model spaces}
\label{model spaces}
We refer to \cite{Lu07} for the following, and restrict ourselves to the case of $G$ simply connected.

A homogeneous space $G/H$ is model for a given semisimple group $G=G_1\times\dots\times G_r$ (with $G_i$ simple for $i=1,\dots,r$)
if and only if $H=H_1\times\dots\times H_r$, with $H_i\subset G_i$, and $G_i/H_i$ is model for $i=1,\dots,r$.

In case of a simple group $G$, the model wonderful variety is obtained as 
the wonderful compactification of $G/H$ with $H$ one of the following subgroups.
\smallbreak
\noindent
-\enspace Type $\mathsf A_n$ with $n$ even: $H=Sp(n)\times GL_1$.
\smallbreak
\noindent
-\enspace Type  $\mathsf A_n$ with $n$ odd: $H$ is the parabolic subgroup of semisimple type $\mathsf C_{(n-1)/2}$ of the symmetric subgroup A~II.
\smallbreak
\noindent
-\enspace Type $\mathsf B_n$ with $n$ even: consider the parabolic of
semisimple type $\mathsf A_{n-1}$ of the symmetric subgroup B~II,
$H$ is included there with the same radical and with semisimple type $\mathsf C_{n/2}$. 
\smallbreak
\noindent
-\enspace Type $\mathsf B_n$ with $n$ odd: $H$ equals the normaliser of the stabiliser of the nilpotent element in the Lie algebra of $G$ with characteristic $(10\dots01)$.
\smallbreak
\noindent
-\enspace Type $\mathsf C_n$ with $n$ even: $H$ is the parabolic subgroup of semisimple type $\mathsf C_{(n/2)-1}\times \mathsf C_{n/2}$ of the symmetric subgroup C~II with $q=2$.
\smallbreak
\noindent
-\enspace Type $\mathsf C_n$ with  $n$ odd: $H$ is the parabolic of semisimple type $\mathsf C_{(n-1)/2}\times \mathsf C_{(n-1)/2}$ of the symmetric subgroup C~II with $q=3$.
\smallbreak
\noindent
-\enspace Type $\mathsf D_n$ with $n$ even: $H$ is the normaliser of the stabiliser of the nilpotent element in the Lie algebra with characteristic $(10\dots011)$. 
\smallbreak
\noindent
-\enspace Type $\mathsf D_n$ with $n$ odd: consider the parabolic subgroup of semisimple type $\mathsf A_{n-2}$ of the symmetric subgroup D~II, 
$H$ is included there with the same radical and with semisimple type $\mathsf C_{(n-1)/2}$.
\smallbreak
\noindent
-\enspace Type $\mathsf E_6$: $H$ is the parabolic subgroup of
semisimple type $\mathsf C_3$ of the symmetric subgroup E~IV.

\smallbreak
\noindent
-\enspace Type $\mathsf E_7$: $H$ is the normaliser of the stabiliser of the nilpotent element in the Lie algebra with characteristic $(010\dots 01)$.
\smallbreak
\noindent
-\enspace Type $\mathsf E_8$: $H$ is the normaliser of the stabiliser of the nilpotent element in the Lie algebra with characteristic $(010\dots 0)$.

\smallbreak
\noindent
-\enspace Type $\mathsf F_4$: $H$ is the parabolic subgroup of semisimple type $\mathsf A_1\times\mathsf B_2$ of the symmetric subgroup F~I.
\smallbreak
\noindent
-\enspace Type $\mathsf G_2$: $H$ is the normaliser of the stabiliser of the nilpotent element in the Lie algebra with characteristic $(10)$.

\bibliographystyle{plain}
\bibliography{2009}

\end{document}